\documentclass[a4paper,12pt,oneside]{amsart}
\usepackage {amscd}
\usepackage {amsmath}
\usepackage {amssymb}
\usepackage [backend=biber,style=alphabetic,maxalphanames=4,sorting=nyt]{biblatex}
\usepackage {mathrsfs}
\usepackage {geometry}
\usepackage {setspace}
\usepackage {tikz}
\usepackage {tikz-cd}
\usetikzlibrary{arrows, decorations.pathmorphing}
\usepackage {float}
\usepackage {indentfirst}
\usepackage {enumitem}
\usepackage {amsthm}
\usepackage {hyperref}
\usepackage[capitalise]{cleveref}
\usepackage {mathabx}
\usepackage {diagbox}
\usepackage {graphicx}
\usepackage{longtable}
\usepackage{makecell}
\usepackage{etoolbox}
\usepackage{multirow}
\usepackage{adjustbox}
\preto\longtable{\setcounter{magicrownumbers}{0}}
\newcounter{magicrownumbers}

\graphicspath{ {./images/} }

\theoremstyle{plain}
\newtheorem{theorem}{Theorem}[section]
\newtheorem*{theorem*}{Theorem}
\newtheorem{proposition}[theorem]{Proposition}

\newtheorem{lemma}[theorem]{Lemma}
\newtheorem{conjecture}[theorem]{Conjecture}

\theoremstyle{definition}
\newtheorem{definition}[theorem]{Definition}
\newtheorem{proposition-definition}[theorem]{Proposition-Definition}

\newtheorem{convention}[theorem]{Convention}

\theoremstyle{remark}
\newtheorem{example}[theorem]{Example}
\newtheorem{construction}[theorem]{Construction}
\newtheorem{statement}[theorem]{}

\newtheorem*{remark}{Remark}

\newcolumntype{C}{>{$}c<{$}} % math-mode version of "l" column type

\title{Toric Landau-Ginzburg models in threefold divisorial contractions}

\author{Yang He\textsuperscript{1}}
\email{heyang@bimsa.cn}
\author{Artan Sheshmani \textsuperscript{1,2,3}}
\email{artan@mit.edu}
\address{\textsuperscript{1}Beijing Institute of Mathematical Sciences and Applications, No. 544, Hefangkou Village, Huaibei Town, Huairou District, Beijing 101408}
\address{\textsuperscript{2}Massachusetts Institute of Technology, IAiFi Institute, 77 Massachusetts Ave, 26-555. Cambridge, MA 02139}
\address{\textsuperscript{3}National Research University, Higher School of Economics, Russian Federation, Laboratory of Mirror Symmetry, NRU HSE, 6 Usacheva str., Moscow, Russia, 119048}

\begin{document}

\begin{abstract}
    We investigate quantum periods and toric Landau-Ginzburg models under divisorial contractions of terminal Fano threefolds. Let $g:Y \rightarrow X$ be a divisorial contraction between $\mathbb{Q}$-factorial Fano threefolds with ordinary terminal singularities and $E$ be the exceptional divisor. Assuming that the center of the contraction is either a smooth point, a terminal quotient point, a point of type cA/n, or a smooth curve with singularities of type cA or cA/n, we prove the regularized period identity
    $$
    \lim_{r\to+\infty}\hat{G}_{Y,rE}(t)=\hat{G}_X(t)
    $$
    where $\hat{G}_{Y,rE}(t)$ and $\hat{G}_X(t)$ are the regularized quantum periods of $(Y,rE)$ and $X$ respectively. This gives a mirror approach to the computation of the Sarkisov links and higher syzygies of central models of dimension 3.
\end{abstract}

\maketitle


\section{Introduction}

    Mirror symmetry predicts that there is a correspondence between the birational operations in the minimal model program (MMP) on Fano varieties and the degenerations of the associated toric Landau-Ginzburg models. More explicitly, we expect the following correspondence:
    \begin{footnotesize}
    \begin{center}
    \begin{longtable}{|C|C|}
    \hline
    \text{The minimal model program operation} &  \text{The (compactified) Landau-Ginzburg models} \\
    \hline
    \textrm{flips/flops} &  \textrm{wall-crossings}\\
    \hline
    \textrm{divisorial contractions} &  \textrm{(divisorial) irreducible degenerations}\\
    \hline
    \textrm{Mori fibre spaces} &  \textrm{(divisorial) fibring degenerations}\\
    \hline
    \end{longtable}
    \end{center}
    \end{footnotesize}
    Following this idea, one can compute the quantum period and the toric Landau-Ginzburg models using MMP, or conversely, compute the outcome of an MMP using toric Landau-Ginzburg models. In \cite{Self_LG_Model}, we studied this principle, and presented the explicit degeneration of the quantum periods and toric Landau-Ginzburg models under a divisorial contraction or a Mori fibre space, starting from a smooth Fano threefold or a smooth toric Fano variety. More explicitly, we have:

    \begin{theorem*}[cf. {\cite[Theorem 1.6]{Self_LG_Model}}]
    \hspace{2em} \\\\
        \emph{Toric LG model under divisorial contractions.} Let $g: X \rightarrow X'$ be either
        \begin{enumerate}
            \item a toric divisorial contraction between smooth toric Fano varieties, or
            \item a divisorial contraction such that $X$ is a smooth Fano surface or threefold with a parametrized toric LG model.
        \end{enumerate}
        Let $f$ be a toric Landau-Ginzburg model of $X$ and $f'$ be a toric Landau-Ginzburg model of $X'$. Then there exists an associated divisorial degeneration from $f$ to $f'$. \\\\
        \emph{Toric LG model under Mori fibrations.} Let $h: X \rightarrow Z$ be either
        \begin{enumerate}
            \item a toric Mori fibration between smooth toric Fano varieties, or
            \item a Mori fibration such that $X$ is a smooth Fano surface or threefold with a parametrized toric LG model.
        \end{enumerate}
        Let $F$ be a general fibre of $h$, $f$ be a toric Landau-Ginzburg model of $X$, and $f'$ be a toric Landau-Ginzburg model of $F$. Then there exists an associated degeneration from $f$ to $f'$.
    \end{theorem*}
    
    It turns out that this approach is powerful in the study of the Sarkisov program and the syzygies of Mori fibre spaces in dimension 3. In fact, the Sarkisov links and higher elementary syzygies for a smooth Fano threefold center are computed in \cite{Self_LG_Model} as an application of the above theorem. The purpose of the present paper is to extend this picture to multiple cases of divisorial contractions between \emph{terminal} Fano threefolds, which were conjectured in \cite{Self_LG_Model}. The main theorems of this article are as follows. First, for divisorial contractions to a point, we have:

\begin{theorem}[also \cref{main theorem 1}]\label{main theorem 1 introduction}
    Let $g: Y \rightarrow X$ be a divisorial contraction such that:
    \begin{enumerate}
        \item $X$ and $Y$ are Fano threefolds with $\mathbb{Q}$-factorial ordinary terminal singularities;
        \item $g$ contracts the exceptional divisor $E$ to a point $P \in X$;
        \item $P \in X$ is a smooth point, or a terminal quotient singularity, or a terminal singularity of type cA/n.
    \end{enumerate}
    Then we have
    $$
    \lim\limits_{r \rightarrow +\infty} \hat{G}_{Y,rE}(t) = \hat{G}_{X}(t)
    $$
    where $\hat{G}_{Y,rE}(t)$ and $\hat{G}_{X}(t)$ are the quantum periods of $(Y,rE)$ and $X$ respectively. In particular, if $Y$ has a parametrized toric Landau-Ginzburg model $\widetilde{f}_{Y}$ and $\lim\limits_{r \rightarrow +\infty} \widetilde{f}_{Y,rE}$ exists, then $X$ has a toric Landau-Ginzburg model $f_{X} = \lim\limits_{r \rightarrow +\infty} \widetilde{f}_{Y,rE}$.
\end{theorem}

    Next, for divisorial contractions to a smooth curve, we have:

\begin{theorem}[\cref{main theorem 2}]\label{main theorem 2 introduction}
    Let $g: Y \rightarrow X$ be a divisorial contraction such that:
    \begin{enumerate}
        \item $X$ and $Y$ are Fano threefolds with $\mathbb{Q}$-factorial ordinary terminal singularities;
        \item $g$ contracts the exceptional divisor $E$ to a smooth curve $C$ on $X$;
        \item $C$ only contains terminal singularities of type cA or cA/n.
    \end{enumerate}
    Then we have
    $$
    \lim\limits_{r \rightarrow +\infty} \hat{G}_{Y,rE}(t) = \hat{G}_{X}(t)
    $$
    where $\hat{G}_{Y,rE}(t)$ and $\hat{G}_{X}(t)$ are the quantum periods of $(Y,rE)$ and $X$ respectively. In particular, if $Y$ has a parametrized toric Landau-Ginzburg model $\widetilde{f}_{Y}$ and $\lim\limits_{r \rightarrow +\infty} \widetilde{f}_{Y,rE}$ exists, then $X$ has a toric Landau-Ginzburg model $f_{X} = \lim\limits_{r \rightarrow +\infty} \widetilde{f}_{Y,rE}$.
\end{theorem}

    We can apply the main theorem from two different perspectives. On the mirror symmetry side, this allows us to compute the toric LG models and quantum periods of some singular $\mathbb{Q}$-Fano threefolds from given Sarkisov links. An example is given in \cref{example: toric LG model of GRDB 40836}, where we use the Sarkisov link given in \cite[Example 6.3.3]{ProkhorovReid2016QFanoIndex2} to compute a toric LG model of the $\mathbb{Q}$-Fano threefold GRDB\#40836. On the birational geometry side, this allows us to describe the outcome of a Sarkisov link, which is not fully understood yet, by computing toric LG models.

\subsection*{Outline of the paper}

    In Section 2, we recall the preliminaries of the paper.
    
    Sections 3 and 4 are the main ingredients of the proof of the main theorems. We briefly explain the sketch of the proof. Consider a terminal divisorial contraction $g: Y \rightarrow X$ and let $E$ be the exceptional divisor.
    
    In Section 3, we construct a degeneration family $\pi:  \mathcal{W} \rightarrow B$ that connects the two varieties $X$ and $Y$. Roughly speaking, a general fibre of $\pi$ is $X$, and the special fibre is a transversal intersection of $Y$ and some other orbifolds along the exceptional divisor $E$. One can think of such a family as a generalization of the ``degeneration to the normal cone'' for the blow-up at smooth centers. This section contains three main parts:
    \begin{enumerate}
        \item First, we use the classification of threefold terminal divisorial contractions to express the divisorial contraction locally as an explicit weighted blow-up.
        \item Second, we use the above explicit weighted blow-up to construct the explicit degeneration family. The key step here is that the pair $(Y,E)$ has a $\mathbb{Q}$-smoothing locally around $E$, so that we can make the central fibre into a transversal intersection between smooth DM stacks after a deformation.
        \item Finally, we globalize the above local construction. The proof is a modification of the arguments in \cite{Sano15}.
    \end{enumerate}
    
    In Section 4, we prove the main theorems by applying the orbifold degeneration formula of \cite{AF16} to the family constructed in Section 3. The main idea of the proof is that the terminal condition places a strong restriction on the age of every twisted sector by the Reid–Shepherd‑Barron–Tai criterion, which controls the virtual dimension of the moduli space of the twisted stable maps. In fact, only one relative Gromov-Witten invariant survives the degeneration formula under the terminal condition.

    In Section 5, we apply the main theorems to some examples.

\subsection*{Acknowledgment}

    The second author is supported by grants from Beijing Institute of Mathematical Sciences and Applications (BIMSA), the Beijing NSF BJNSF-IS24005, and the China National Science Foundation (NSFC) NSFC-RFIS program W2432008. He would also like to thank China's National Program of Overseas High Level Talent for generous support. Finally, he would like to thank NSF AI Institute for Artificial Intelligence and Fundamental Interactions at Massachusetts Institute of Technology (MIT) which is funded by the US NSF grant under Cooperative Agreement PHY-2019786.

\section{Preliminaries}

    In this section, we set up the preliminaries of this article. We refer to \cite{Self_LG_Model} for a more detailed description of the background results.

\subsection{The orbifold Gromov-Witten theory}

    In this subsection, we recall some basic results about the orbifold Gromov-Witten theory, including the orbifold degeneration formula. The standard references in this topic are \cite{ChenRuan2001}, \cite{AGV08}, and \cite{AF16}. We will see in the next subsection that to study quantum periods and toric Landau-Ginzburg models of terminal Fano threefolds, one should reduce to the case of Fano threefolds with terminal quotient singularities, which admit an orbifold structure.
    
\begin{convention}
    Throughout this article, a \emph{(smooth projective) orbifold} is a smooth Deligne-Mumford stack over $\mathbb{C}$ whose coarse moduli space is projective. We say that an orbifold is Fano (resp. weak Fano) if the coarse moduli space is a Fano (resp. weak Fano) variety and is terminal if the coarse moduli space has terminal singularities.

    Let $X$ be a normal variety with only quotient singularities. We denote by $\mathfrak{X}$ its canonical stack. Thus $\mathfrak{X}$ is a smooth Deligne-Mumford stack with coarse moduli space $X$, and the morphism $\mathfrak{X} \rightarrow X$ is an isomorphism over the smooth locus of \(X\). Whenever a variety with quotient singularities appears as a target of the orbifold Gromov-Witten theory, we implicitly replace it by its canonical stack. If $D$ is a Weil divisor on $X$ which becomes Cartier on the local index-one covers, then the pullback of $D$ on $\mathfrak{X}$ is a Cartier divisor. We use the same letter $D$ for the pullback divisor on the canonical stack when there is no confusion.
\end{convention}

\begin{definition}[Inertia stacks]
    Let $\mathfrak{X}$ be a smooth Deligne-Mumford stack. Its inertia stack is
    $$
    \mathcal{I}\mathfrak{X} = \mathfrak{X} \times_{\mathfrak{X}\times \mathfrak{X}}\mathfrak{X}.
    $$
    Equivalently, an object of $\mathcal{I}\mathfrak{X}$ is a pair $(x,g)$, where $x \in \mathfrak{X}$ and \(g\in \operatorname{Aut}(x)\).  We also use the cyclotomic inertia stack
    $$
    \mathcal{I}_\mu(\mathfrak{X}) = \coprod_{r\ge 1}\underline{\operatorname{Hom}}^{\mathrm{rep}} (B\mu_r,\mathfrak{X}),
    $$
    and its rigidification $\overline{\mathcal{I}}_\mu\mathfrak{X}$.

    If $\Omega \subset \mathcal{I}_\mu(\mathfrak{X})$ is a connected component and \((x,g)\in \Omega\), write the eigenvalues of $g$ on $T_x\mathfrak{X}$ as
    $$
    \exp(2\pi i a_1),\cdots,\exp(2\pi i a_n), \qquad 0 \le a_j<1.
    $$
    The \emph{age} of $\Omega$ is
    $$
        \mathrm{Age}_{\mathfrak{X}}(\Omega)=\sum_j a_j.
    $$
    This convention is the one used in the virtual-dimension formula for orbifold Gromov--Witten invariants.
\end{definition}

\begin{definition}[Twisted curves and twisted stable maps]
    A \emph{twisted curve} is a Deligne-Mumford stack $\mathfrak{C}$ whose coarse moduli space $C$ is a nodal marked curve, such that the stack structure is allowed only at marked points and nodes. Locally at a marked point, it has the form
    $$
        [\mathrm{Spec }\mathbb{C}[z]/\mu_r],
        \qquad
        \zeta\cdot z=\zeta z,
    $$
    and locally at a node, it has the balanced form
    $$
        [\mathrm{Spec }\mathbb{C}[x,y]/(xy)/\mu_r],
        \qquad
        \zeta\cdot(x,y)=(\zeta x,\zeta^{-1}y).
    $$
    A \emph{twisted stable map} to $\mathfrak{X}$ is a representable morphism
    $$
        f:\mathfrak{C} \rightarrow \mathfrak{X}
    $$
    from a twisted curve, satisfying the usual stability condition on the induced map of coarse spaces. The evaluation maps of twisted stable maps land in \(\overline{\mathcal{I}}_{\mu}\mathfrak{X}\).
\end{definition}

\begin{definition}[The moduli space of twisted stable maps]
    Let $\mathfrak{X}$ be an orbifold with coarse moduli space $X$. Let $g,n$ be non-negative integers and $\beta$ be a curve class on $X$. The \emph{moduli stack of twisted stable maps to $\mathfrak{X}$ of genus $g$ curves of class $\beta \in H_{2}(X,\mathbb{Z})$ with $n$ marked points} (we denote it by $\mathcal{K}_{g,n}(\mathfrak{X},\beta)$) is the Deligne-Mumford stack of twisted stable maps $f: \mathcal{C} \rightarrow \mathfrak{X}$ of curves of genus $g$ with $n$ marked points such that the underlying map satisfies $f_{*}C = \beta$.
\end{definition}

\begin{definition}[Gromov-Witten invariants with gravitational descendants, cf. \protect{\cite[\S 8.3]{AGV08}}, \protect{\cite[\S 2.1]{Iritani20}}]

    On $\mathcal{K}_{g,n}(X,\beta)$, one defines $n$ tautological line bundles $\mathcal{L}_{i}$ whose fibre at a point is the cotangent space to the corresponding \emph{coarse} curve at the $i$-th marked point. The \emph{cotangent line class} is the class
    $$
    \psi_{i}= c_{1}(\mathcal{L}_{i}) \in H^{2}(\mathcal{K}_{g,n}(X,\beta)).
    $$
    
    Let $\overline{\mathcal{I}}\mathfrak{X}$ be the rigidified cyclotomic inertia stack of $\mathfrak{X}$. We consider the evaluation maps 
    \begin{align*}
        ev_{i}: \mathcal{K}_{g,n}(\mathfrak{X},\beta) &\rightarrow \overline{\mathcal{I}}\mathfrak{X},\\
        ev_{i}(C,\mu;p_{1}, \cdots , p_{n},f) &= (f(p_{i}),\mu).
    \end{align*}
    Consider $\gamma_{1}, \cdots, \gamma_{n} \in H^{*}(\overline{\mathcal{I}}\mathfrak{X})$. Let $a_{1},\cdots, a_{n}$ be non-negative integers, and let $\beta \in H_{2}(X)$. Then the \emph{Gromov-Witten invariant with descendants} is the number given by
    $$
    \langle \tau_{a_{1}}\gamma_{1}, \cdots , \tau_{a_{n}}\gamma_{n} \rangle_{\beta} = \int_{[\mathcal{K}_{g,n}(\mathfrak{X},\beta)]^{virt}} \psi_{1}^{a_{1}}ev_{1}^{*}(\gamma_{1}) \cdots \psi_{n}^{a_{n}}ev_{n}^{*}(\gamma_{n}).
    $$
\end{definition}

\begin{convention}[Notation for the smooth case]
   When $X$ is a smooth variety, the moduli space of stable maps will be denoted by $\overline{M}_{g,n}(X,\beta)$.
\end{convention}

\begin{statement}[Set-up]\label{setup:I-series}
        During this subsection, we set up the following data:
    \begin{enumerate}
        \item Let $\mathfrak{X}$ be a terminal Fano orbifold and $X$ be its coarse moduli space. In particular, $X$ is a $\mathbb{Q}$-factorial Fano threefold with terminal quotient singularities. Let $n$ be the dimension of $X$ and $r$ be the Picard number of $X$.
        \item Let $D$ be a $\mathbb{C}$-divisor on $X$.
        \item Let $K \subseteq H_{2}(X,\mathbb{Z})$ be the monoid of classes of moving curves of $X$, that is, $K$ consists of classes $\beta \in H_{2}(X,\mathbb{Z})$ such that the morphism $ev_{1}: \mathcal{K}_{0,1}(\mathfrak{X},\beta) \rightarrow X$ is surjective. In particular, for every class $0 \neq \beta \in K$ we have $K_{X} \cdot \beta < 0$.
        \item Let 
        $$
        \hat{G}_{X}(t) = \widetilde{I}^{X}_{0}(t) = 1 + \sum\limits_{\beta \in K}(-K_{X} \cdot \beta)!\langle \tau_{-K_{X}\cdot\beta -2}\mathbf{1} \rangle_{\beta} \cdot t^{-K_{X} \cdot \beta}
        $$
        be \emph{the constant term of regularized $I$-series} or \emph{the regularized quantum period} of $X$. Here the class $\mathbf{1} \in H^{2 \mathrm{dim}X}(\overline{\mathcal{I}}\mathfrak{X},\mathbb{Q})$ is the Poincaré dual of the point class of the untwisted sector of $\overline{\mathcal{I}}\mathfrak{X}$.
        \item Let 
        \begin{align*}
        \hat{G}_{X,D}(t) = \widetilde{I}_{0}^{X,D}(t) &= 1 + \sum\limits_{\beta \in K}(-K_{X} \cdot \beta)!\langle \tau_{-K_{X}\cdot\beta -2}\mathbf{1} \rangle_{\beta} \cdot e^{-D \cdot \beta}t^{-K_{X} \cdot \beta}\\
        &= 1 + a_{1}t + a_{2}t^2 + \cdots
        \end{align*} 
        be \emph{the restriction of the constant term of regularized $I$-series of $X$ to the
        anti-canonical direction corresponding to $D$}.
    \end{enumerate}
\end{statement}

\begin{remark}
    The above definition of the quantum period is a special case of the definition given in \cite{OnetoPetracci2018}. Indeed, by \cref{thm:MS-2.3} below, terminal quotient singularities have ages $>1$, so there is no insertion coming from sectors with ages $<1$.
\end{remark}

    Now we give the definitions of Landau-Ginzburg models.

\begin{definition}[Weak Landau-Ginzburg models]
    A \emph{weak Landau-Ginzburg model of $X$} is a Calabi-Yau fibration $f: Y \rightarrow \mathbb{A}^1$ satisfying the period condition, i.e., the period of $f$ corresponds to the $I$-series of $X$. In particular, the anti-canonical sections of $X$ are mirrored to the fibres of $f$.
\end{definition}

    The most important and practically calculable cases of Landau-Ginzburg models are the toric Landau-Ginzburg models, which are defined as follows:

\begin{definition}[Toric Landau-Ginzburg models]
    A \emph{toric Landau-Ginzburg model} of $(X,D)$ is a Laurent polynomial $f: (\mathbb{C}^{*})^{n} \rightarrow \mathbb{C}$ such that
    \begin{enumerate}
        \item \emph{Period condition.} $\hat{P}_{f}(t) = \hat{G}_{X,D}(t)$.
        \item \emph{Calabi-Yau compactification. } There exists a fiberwise compactification (the so called \emph{Calabi–Yau compactification}) $Y \rightarrow \mathbb{C}$ such that $Y$ is a smooth Calabi-Yau variety.
        \item \emph{Polytope condition. }There is a degeneration $X \leadsto X_{T}$ to a toric variety $X_{T}$ whose fan polytope (the convex hull of generators of its rays) coincides with the Newton polytope (the convex hull of non-zero coefficients) of $f$.
    \end{enumerate}

    A Laurent polynomial $f$ is said to be a \emph{divisorial toric Landau-Ginzburg model of $X$} if there exists a $\mathbb{C}$-divisor $D$ such that $f$ is a toric LG model of $(X,D)$.
\end{definition}
    
    Now we recall the orbifold degeneration formula. We start with the definition of transversal intersections for orbifolds.

\begin{definition}[Transversal intersection of smooth Deligne-Mumford stacks, cf. \protect{\cite[Appendix A]{AF16}}]

\hspace{2em}

\begin{itemize}
    \item Let $X$ be a complex algebraic stack. We say that $X$ has \emph{nodal (codimension 1) singularities} if it is locally isomorphic in the f.p.p.f. topology to $\mathrm{Spec}\frac{\mathbb{C}[x,y,z_{1},\cdots,z_{n}]}{(xy)}$. In particular, the singular locus $D$ of $X$ is a smooth substack. Let $\nu: \widetilde{X} \rightarrow X$ be the normalization of $X$ and $\widetilde{D} = \nu^{-1}D$.
    \item A morphism $f: C \rightarrow X$ between nodal algebraic stacks is called \emph{transversal to the singular locus} if
    \begin{enumerate}
        \item the induced morphism $\widetilde{C} \rightarrow \widetilde{X}$ defines a morphism of locally smooth pairs which is transversal to the boundary divisor;
        \item for every point $p \in f^{-1}D$ its two inverse images in $\widetilde{C}$ map to different points of $\widetilde{D}$ via $f$.
    \end{enumerate}
\end{itemize}
\end{definition}

    We also need the relative orbifold Gromov-Witten theory for the orbifold degeneration formula. First, we set up the data we need to define the relative orbifold Gromov-Witten theory.

\begin{convention}[Data for a pair]\label{data for a pair}
    Fix a smooth orbifold pair $(X,D)$. We also fix
    \begin{enumerate}
        \item a curve class $\beta$ on $X$;
        \item an integer $g \geq 0$;
        \item disjoint finite ordered sets $N,M$ which we may take to be $\{1,2,\cdots,n \}$ and $\{n+1,n+2,\cdots,n + |M|\}$;
        \item tuples $\mathbf{e} = (e_{i})_{i \in N}$ and $\mathbf{f} = (f_{j})_{j \in M}$ of positive integers such that $\mathcal{I}_{e_{i}}(X) \neq \emptyset$ for all $i \in N$ and $\mathcal{I}_{f_{j}}(D) \neq \emptyset$ for all $j \in M$.
        \item a tuple $\mathbf{c} = (c_{j})_{j \in M}$ of positive integers such that $\sum\limits_{j \in M}\frac{c_{j}}{f_{j}} = D \cdot \beta$.
    \end{enumerate}
    We denote by $\boldsymbol\mu = (\mu_{j})_{j \in M}$ the tuple formed by $\mu_{j} = \frac{c_{j}}{f_{j}}$.
\end{convention}

\begin{definition}[Relative Gromov-Witten invariants with gravitational descendants, cf. {\cite[Definition 4.5]{AF16}}]
    Notation as in \cref{data for a pair}. Let $m_{i},i \in N$ be non-negative integers, $\gamma_{i} \in H^{*}(\overline{\mathcal{I}}(X)), i \in N$ and $\gamma_{j} \in H^{*}(\overline{\mathcal{I}}(D)), j \in M$. We define relative Gromov-Witten invariants with gravitational descendants by the formula
    $$
    \langle \prod\limits_{i \in N} \tau_{m_{i}}(\gamma_{i}) \mid \prod\limits_{j \in M} \gamma_{j} \rangle_{\Gamma}^{(X,D)} := \mathrm{deg}((\prod\limits_{i \in N}\psi_{i}^{m_{i}}\cdot \mathrm{ev}_{i}^{*}\gamma_{i}) \cdot (\prod\limits_{j \in M}\mathrm{ev}_{j}^{*}\gamma_{j}) \cap [\mathcal{K}_{\Gamma}^{\mathfrak{r}}(X,D)]^{vir}).
    $$
\end{definition}

    Now we are ready to state the orbifold degeneration formula.

\begin{convention}[Data for a degeneration]\label{data for a degeneration}
    In this subsection, we fix a flat morphism $\pi: W \rightarrow B$ such that $B$ is a smooth curve, $W$ is a smooth Deligne-Mumford stack, and $0 \in B$ is the unique critical value of $\pi$. We assume that $W_{0} = X_{1} \cup_{D} X_{2}$ is the union of two smooth closed substacks $X_{1}$ and $X_{2}$ intersecting transversally along $D$, a smooth divisor on both $X_{1}$ and $X_{2}$. This implies that $W_{0}$ is nodal and first-order smoothable along the singular locus $D$. The rigidified inertia stack of $W_{0}$ is denoted by $\overline{\mathcal{I}}(W_{0})$. We also fix
    \begin{enumerate}
        \item a curve class $\beta$ in the fibre of $\pi$;
        \item an integer $g \geq 0$;
        \item a finite ordered set $N$, possibly empty, which we may take to be $\{1,2,\cdots,n \}$;
        \item a tuple $\mathbf{e} = (e_{i})_{i \in N}$ of positive integers such that $\mathcal{I}_{e_{i}}(W_{b}) \neq \emptyset$ for all $b \in B, i \in N$. 
    \end{enumerate}
\end{convention}

\begin{theorem}[The orbifold degeneration formula, {\cite[\S 5]{AF16}}]\label{theorem: orbifold degeneration formula}
    Notation as in \cref{data for a degeneration}. Assume that $W_{0}$ is a proper Deligne-Mumford stack having a projective coarse moduli space. Let $\gamma_{1},\cdots,\gamma_{n} \in H^{*}_{orbifold}(W_{0},\mathbb{Q}) := H^{*}(\overline{\mathcal{I}}(W_{0}),\mathbb{Q})$ be classes having homogeneous parity, and $a_{1},\cdots,a_{n}$ be non-negative integers. Then we have
    \begin{align*}
        \langle \prod\limits_{i=1}^{n}\tau_{a_{i}}(\gamma_{i}) \rangle^{W_{0}}_{g,\beta} = \sum\limits_{\eta \in \Omega} \frac{\prod\limits_{j \in M}d_{j}}{|M|!}\sum\limits_{\substack{\delta_{j} \in F \\ j \in M}}(-1)^{\epsilon} \langle \prod\limits_{i \in N_{1}}\tau_{a_{i}}(\gamma_{i}) \mid \prod\limits_{j \in M} \delta_{j} \rangle^{(X_{1},D)}_{\Xi_{1}} \cdot \langle \prod\limits_{i \in N_{2}}\tau_{a_{i}}(\gamma_{i}) \mid \prod\limits_{j \in M} \widetilde{\delta}^{\vee}_{j} \rangle^{(X_{2},D)}_{\Xi_{2}}
    \end{align*}
    where
    \begin{enumerate}
        \item $F$ is a homogeneous basis of $H^{*}(\overline{\mathcal{I}}(D),\mathbb{Q})$.
        \item $\widetilde{\delta}^{\vee}_{j}$ is the dual of $\delta_{j}$ with respect to the Chen-Ruan pairing.
        \item $\Omega$ is the set of splittings of the data $g,n,\beta$. An element in $\Omega$ includes the data below:
        \item $N_{1},N_{2}$ is a decomposition of $\{1,2,\cdots,n \}$ in two subsets.
        \item $\Xi_{1},\Xi_{2}$ is a possibly disconnected splitting of the data $\beta,g$ in two modular graphs having roots labeled by $M:=\{n+1,n+2,\cdots,n+|M|\}$.
        \item $d_{i}$ are assigned intersection multiplicities for $i \in M$.
        \item $(-1)^{\epsilon}$ is the sign determined formally by the equality
        $$
        \prod\limits_{i=1}^{n} \gamma_{i} \prod\limits_{j \in M}\delta_{j}\widetilde{\delta}^{\vee}_{j} = (-1)^{\epsilon} \prod\limits_{i \in N_{1}} \gamma_{i} \prod\limits_{j \in M}\delta_{j} \prod\limits_{i \in N_{2}}\gamma_{i} \prod\limits_{j \in M} \widetilde{\delta}^{\vee}_{j}.
        $$
    \end{enumerate}
\end{theorem}

\subsection{Terminal Fano threefolds}

    In this subsection, we recall some results about terminal Fano threefolds. First, we recall the definition of terminal singularities.

\begin{definition}[Terminal singularities]
    Let $(X,\Delta)$ be a pair where $X$ is a normal variety and $\Delta = \sum a_{i}D_{i}$ is an $\mathbb{R}$-linear combination of distinct prime divisors such that $0 \leq a_{i} \leq 1$ for all $i$. Assume that $K_{X} + \Delta$ is $\mathbb{R}$-Cartier. Let $f: Y \rightarrow X$ be a log resolution of $(X,\Delta)$. Write
    $$
    K_{Y} + f^{-1}_{*}\Delta = f^{*}(K_{X}+\Delta) + \sum\limits_{E_{i} \text{ is } f-\text{exceptional}} a(E_{i};X,\Delta) E_{i}.
    $$
    We say that the pair $(X,\Delta)$ is \emph{terminal} if $a(E_{i};X,\Delta)>0$ for every exceptional divisor $E_{i}$ of $f$. We say that $X$ is terminal if $(X,0)$ is terminal.
\end{definition}

    Now, we recall the classification of threefold terminal singularities.

\begin{theorem}[Classification of threefold terminal singularities, I, cf. \cite{Reid2}]
    Let $X$ be a germ of a 3-dimensional Gorenstein terminal singularity. Then $X$ is an isolated compound Du Val (cDV) singularity.
\end{theorem}

\begin{theorem}[Classification of threefold terminal singularities, II, cf. \protect{\cite{Mori1985}}, \cite{KSB}]\label{thm: Classification of threefold terminal singularities II}
    Let $X$ be a germ of a 3-dimensional terminal singularity of index $\geq 2$. Then there is an embedding $j: X \rightarrow \mathbb{A}^{4}/\mathbb{Z}_{m}$ such that the image is one of the following:
    \begin{enumerate}[align=left]
        \item [(cA/m)] $(xy + f(z,u)) \subseteq \mathbb{A}^{4}/\frac{1}{m}(1,-1,b,0)$, where $b$ is an integer prime to $m$ and $f$ is a $\mathbb{Z}_{m}$-invariant. This includes the case of quotient singularities $\mathbb{A}^3/\frac{1}{m}(1,-1,b)$, which we will denote by (Q).
        \item [(cAx/4)] $(x^{2} + y^{2} + f(z,u)) \subseteq \mathbb{A}^{4}/\frac{1}{4}(1,3,1,2)$, where $f$ is $\mathbb{Z}_{4}$-semi-invariant and $u \notin f(z,u)$.
        \item [(cAx/2)] $(x^{2} + y^{2} + f(z,u)) \subseteq \mathbb{A}^{4}/\frac{1}{2}(0,1,1,1)$, where $f$ is $\mathbb{Z}_{2}$-invariant and has order $\geq 4$.
        \item [(cD/3-1)] $(u^{2} + x^{3} + yz(y+z)) \subseteq \mathbb{A}^{4}/\frac{1}{3}(1,2,2,0)$.
        \item [(cD/3-2)] $(u^{2} + x^{3} + yz^{2} + xy^{4}\lambda(y^{3})+y^{6}\mu(y^{3})) \subseteq \mathbb{A}^{4}/\frac{1}{3}(1,2,2,0)$ where $4\lambda^{3}+27\mu^{2} \neq 0$.
        \item [(cD/3-3)] $(u^{2} + x^{3} + y^{3} + xyz^{3}\alpha(z^{3}) + xz^{4}\beta(z^{3})+yz^{5}\gamma(z^{3})+z^{6}\delta(z^{3})) \subseteq \mathbb{A}^{4}/\frac{1}{3}(1,2,2,0)$.
        \item [(cD/2-1)] $(u^{2}+xyz+x^{2a}+y^{2b}+z^{c}) \subseteq \mathbb{A}^{4}/\frac{1}{2}(1,1,0,1)$ where $a,b \geq 2$, $c \geq 3$.
        \item [(cD/2-2)] $(u^{2}+y^{2}z+\lambda x^{2a+1}y + g(x,z)) \subseteq \mathbb{A}^{4}/\frac{1}{2}(1,1,0,1)$ where $g(x,z) \in (x^{4},x^{2}z^{2},z^{3})\mathbb{C}\{x,z\}$.
        \item [(cE/2)] $(u^{2}+x^{3}+xg(y,z)+h(y,z)) \subseteq \mathbb{A}^{4}/\frac{1}{2}(0,1,1,1)$, where $g$ has order $\geq 4$ and $h$ has order $4$.
    \end{enumerate}
    Conversely, every isolated singularity of the above form is a terminal singularity.
\end{theorem}

    One major issue for singular varieties is that we don't yet have a well-defined Gromov-Witten theory. Therefore, we will only study the varieties which admit a $\mathbb{Q}$-Gorenstein deformation to varieties with only quotient singularities, so that we can apply the orbifold Gromov-Witten theory. Such a deformation is called a $\mathbb{Q}$-smoothing.

\begin{definition}[The $\mathbb{Q}$-smoothing]
    Let $X$ be a $\mathbb{Q}$-Gorenstein variety. A $\mathbb{Q}$-smoothing of $X$ is a $\mathbb{Q}$-Gorenstein family $\mathcal{X} \rightarrow C \ni 0$ such that:
    \begin{itemize}
        \item $\mathcal{X}_{0} \cong X$;
        \item $\mathcal{X}_{t}$ has quotient singularities for $t \neq 0$.
    \end{itemize}
\end{definition}

    In this article, we will focus on a slightly smaller class of terminal singularities, namely ordinary terminal singularities. We will soon see that this condition guarantees the existence of a $\mathbb{Q}$-smoothing. To simplify the definition, we will use the classification of threefold terminal singularities to define ordinary terminal singularities. One can find other interpretations in \cite[p. 549]{Morrison1986RemarkKawamata} and \cite[Definition 3.2]{Sano15}.

\begin{definition}[Ordinary terminal singularities]
    The terminal singularities of type cDV, cA/n, cAx/2, cD/2, cD/3, cE/2 are called \emph{ordinary terminal singularities}. The terminal singularities of type cAx/4 are called \emph{exceptional (non-ordinary) terminal singularities}.
\end{definition}

    The following result about $\mathbb{Q}$-smoothing is the important ingredient for our definition:

\begin{theorem}[cf. \protect{\cite[Theorem 1.5]{Sano15}}]\label{thm: Q-smoothing of Fano threefolds with ordinary terminal singularities}
    Let $X$ be a Fano threefold with ordinary terminal singularities, then $X$ has a $\mathbb{Q}$-smoothing.
\end{theorem}

    In fact, it is conjectured that $\mathbb{Q}$-smoothings exist for arbitrary terminal Fano threefolds.

\begin{conjecture}[cf. \protect{\cite[Conjecture 1.4]{Sano15}}, \protect{\cite[\S 4.8.3]{ABR02}}]
     Let $X$ be a Fano threefold with terminal singularities, then $X$ has a $\mathbb{Q}$-smoothing.
\end{conjecture}

    In the same article, Sano also proves the following:
    
\begin{theorem}[cf. \protect{\cite[Theorem 1.7]{Sano15}}]\label{thm: Unobstructed deformations of terminal Fano threefolds}
    Let $X$ be a Fano threefold with terminal singularities, then the deformations of $X$ are unobstructed.
\end{theorem}

    In this article, we will need the notion of the $\mathbb{Q}$-smoothing of a threefold pair $(X,D)$. Note that our definition here is different from the ``simultaneous $\mathbb{Q}$-smoothing with an elephant'' defined in \cite{ABR02}
    
\begin{definition}
    Let $X$ be a terminal Fano threefold, and $D$ be an effective Weil divisor on $X$ such that for every point $P \in D$, $D$ becomes Cartier at the index-one cover of $P \in X$. A $\mathbb{Q}$-smoothing of $(X,D)$ is a $\mathbb{Q}$-Gorenstein family $\mathcal{X} \rightarrow C \ni 0$ and a flat effective Weil divisor $\mathcal{D}$ on $\mathcal{X}$ such that:
    \begin{itemize}
        \item for every $P \in \mathcal{D}$, $\mathcal{D}$ becomes Cartier at the index-one cover of $P \in \mathcal{X}$;
        \item $(\mathcal{X}_{0},\mathcal{D}_{0}) \cong (X,D)$;
        \item for $t \neq 0$, $(\mathcal{X}_{t},\mathcal{D}_{t})$ is locally isomorphic to
        $$
        (\mathbb{A}^{n}/G, (x=0))
        $$
        for some finite group $G$.
    \end{itemize}
\end{definition}

    We can define Gromov-Witten theory for Fano threefolds that admit a $\mathbb{Q}$-smoothing.

\begin{definition}
    Let $X$ be a Fano threefold with terminal singularities. Assume that there exists a $\mathbb{Q}$-smoothing $\mathcal{X} \rightarrow C \ni 0$ of $X$. We define the \emph{quantum period of $X$} as the quantum period of $\mathcal{X}_{t}$ for a general $t \in C$. Similarly, let $D$ be an effective Weil divisor on $X$ such that for every point $P \in D$, $D$ becomes Cartier at the index-one cover of $P \in X$. Assume that there exists a $\mathbb{Q}$-smoothing $(\mathcal{X},\mathcal{D}) \rightarrow C \ni 0$ of $X$. We define the \emph{quantum period of $(X,D)$} as the quantum period of $(\mathcal{X}_{t},\mathcal{D}_{t})$ for a general $t \in C$.
\end{definition}

    Notice that for a terminal Fano threefold $X$ and an effective Weil divisor $D$ such that for every point $P \in D$, $D$ becomes Cartier at the index-one cover of $P \in X$, any two $\mathbb{Q}$-smoothings of the pair $(X,D)$ only differ by a locally trivial deformation. Hence the quantum period of $(X,D)$ is independent of the choice of the $\mathbb{Q}$-smoothing and is well-defined. In particular, the above results allow us to define the Gromov-Witten invariant for Fano threefolds with ordinary terminal singularities. 
    
    At the end of this subsection, we recall the Reid–Shepherd‑Barron–Tai criterion of terminal cyclic quotient singularities. This explains why we need terminal quotient singularities rather than arbitrary quotient singularities: the terminal condition controls the age, and the age controls the virtual dimension of the moduli space of twisted curves.

\begin{theorem}[cf. {\cite[Theorem 2.3]{MorrisonStevens1984}}]\label{thm:MS-2.3}
    Let $\Gamma \subset \mathrm{GL}(n,\mathbb{C})$ be a finite \emph{small} subgroup (i.e.\ it contains no quasi-reflections), and set $X := \mathbb{C}^{n}/\Gamma$. For $g\in\Gamma$ of order $M>1$ and a primitive $M$-th root of unity $\zeta$, write the eigenvalues of $g$ as $\zeta^{a_1},\dots,\zeta^{a_n}$ with $0\le a_i < M$, and define
    $$
    e(g,\zeta)\;:=\;\frac{a_1}{M}+\cdots+\frac{a_n}{M}.
    $$
    Then:
    \begin{enumerate}
        \item[\textup{(i)}] (Khinich and Watanabe) $X$ is Gorenstein if and only if $ \Gamma \subset \mathrm{SL}(n,\mathbb{C})$.
        \item[\textup{(ii)}] (Reid, Shephard-Barron, and Tai) $X$ is canonical if and only if $ e(g,\zeta)\ge 1$ for every $g\neq 1$ and every primitive $\zeta$.
        \item[\textup{(iii)}] (Reid) $X$ is terminal if and only if $e(g,\zeta)> 1$ for every $g\neq 1$ and every primitive $\zeta$.
    \end{enumerate}
\end{theorem}

    The above criterion (iii) can be rephrased as, in the language of smooth DM stacks, that a cyclic quotient singularity $P \in X$ is terminal if and only if $\mathrm{Age}(\delta) > 1$ for every twisted sector $\delta$ at $P$.

\subsection{Deformation of divisorial contractions}

    We recall the following classical result about how extremal contractions behave under deformations.
    
\begin{theorem}[cf. \protect{\cite[Theorem 12.3.1]{KM}}]\label{thm: deformation of extremal contractions}
    Let $g: Y \rightarrow S$ be a proper flat morphism of complex spaces. Assume that for some $0 \in S$ the fibre $Y_{0}$ is a projective variety with only $\mathbb{Q}$-factorial rational singularities, $\mathrm{dim}Y_{0} \geq 3$. Let $f_{0}: Y_{0} \rightarrow Z_{0}$ be the contraction of an extremal ray $C_{0} \subseteq Y_{0}$. Then there is a proper flat morphism $Z \rightarrow S$ and a factorization
    $$
    g: Y \xrightarrow{f} Z \rightarrow S.
    $$
    There is an open neighborhood $0 \in U \subseteq S$ such that if $f_{0}$ contracts a subset of codimension at least two (resp. contracts a divisor, resp. is a fibre space of generic relative dimension $k$) then $f_{s}$ contracts a subset of codimension at least two (which may be empty) (resp. contracts a divisor, resp. is a fibre space of generic relative dimension $k$) if $s \in U$.
\end{theorem}

    In particular, we will apply this theorem to study the divisorial contractions under $\mathbb{Q}$-smoothings.

\subsection{The classification of divisorial contractions of terminal threefolds}

    In this subsection, we recall the classification of some threefold divisorial contractions.
    
\begin{theorem}[Classification of threefold divisorial contractions to smooth points, cf. \cite{Kawakita2001}]\label{Classification of 3-fold divisorial contractions to smooth points}
    Let $f: Y \rightarrow X$ be a divisorial contraction between $\mathbb{Q}$-factorial terminal threefolds which contracts the exceptional divisor $E$ to a smooth point $P \in X$. Then we can take local parameters $x,y,z$ at $P$ and coprime positive integers $a$ and $b$, such that $f$ is the weighted blow-up of $X$ with weights $w(x,y,z) = (1,a,b)$.
\end{theorem}

\begin{theorem}[Classification of threefold divisorial contractions to terminal quotient singularities, cf. \cite{Kawamata1996}]\label{Classification of 3-fold divisorial contractions to terminal quotient points}
    Let $f: Y \rightarrow X$ be a divisorial contraction between $\mathbb{Q}$-factorial terminal threefolds such that the image $f(E)$ of the exceptional divisor $E$ contains a terminal quotient singularity $P \in X$. Then we can take local parameters $x,y,z$ at $P$ such that $P \in X$ is of the form $\mathbb{A}^3/\frac{1}{n}(s,-s,1)$ and $f$ is the weighted blow-up of $X$ with weights $w(x,y,z) = \frac{1}{n}(s,n-s,1)$. In particular, we have $P = f(E)$.
\end{theorem}

\begin{theorem}[Classification of threefold divisorial contractions to terminal singularities of type cA/n, cf. \cite{Kawakita2005}]\label{Classification of 3-fold divisorial contractions to cA/n points}
    Let $f: Y \rightarrow X$ be a divisorial contraction between $\mathbb{Q}$-factorial terminal threefolds which contracts the exceptional divisor $E$ to a terminal singularity $P \in X$ of type cA/n for some integer $n \geq 2$. Then we can take local parameters $x,y,z,w$ at $P$ such that $P \in X$ is locally of the form
    $$
    (xy + g(z^{n},w) = 0) \in \mathbb{A}^{4}/\frac{1}{n}(1,-1,b,0)
    $$
    and $f$ is given by the weighted blow-up of weights $\frac{1}{n}(w_{1},w_{2},a,n)$ which satisfies the following conditions:
    \begin{enumerate}
        \item $a \equiv bw_{1} \mod n$ and $w_{1}+w_{2}\equiv 0 \mod an$.
        \item $\frac{a-bw_{1}}{n}$ is coprime to $w_{1}$.
        \item $g$ has weighted order $\frac{w_{1}+w_{2}}{n}$ with weights $\mathrm{wt}(z,w) = (\frac{a}{n},1)$.
        \item The monomial $z^{\frac{w_{1}+w_{2}}{a}}$ appears in $g$ with non-zero coefficient.
    \end{enumerate}
    Moreover, any such $f$ is a divisorial contraction.
\end{theorem}

\begin{theorem}[Classification of threefold divisorial contractions to curves, cf. \cite{CCC24}]\label{Classification of 3-fold divisorial contractions to curves}
    Let $f: Y \supseteq E \rightarrow X \ni P$ be a divisorial contraction between terminal threefolds which contracts the exceptional divisor $E$ to a smooth curve $C \ni P$. Assume that $Y$ has terminal quotient singularities near $E$, and $P \in X$ is a terminal singularity of type cA or cA/n. Then one of the following holds:
    \begin{enumerate}
        \item We can choose local coordinates at $P$ such that $X$ is given by
        $$
        (xy+h_{+}y+g(x,z,w) = 0) \subseteq \mathbb{A}^{4},
        $$
        $C$ is the $x$-axis and $f$ is the weighted blow-up of weight $(0,m,1,1)$.
        \item We can choose local coordinates at $P$ such that $X$ is given by 
        $$
        \left(\begin{cases}
        xu+h_{+}y+g_{>1}(x,z,w) = 0 \\
        u - (x^{k-1}y+z) = 0        
        \end{cases} \right) \subseteq \mathbb{A}^{5},
        $$
        $C$ is the $x$-axis and $f$ is the weighted blow-up of weight $(0,m'-1,1,1,m')$.
        \item We can choose local coordinates at $P$ such that $X$ is given by
        $$
        (xy+h_{+}y+g_{m}(x,z,w) + g_{>m}(x,z,w) = 0) \subseteq \mathbb{A}^{4}/\frac{1}{n}(-1,1,b,0),
        $$
        $C$ is the $x$-axis and $f$ is the weighted blow-up of weight $(0,m,1,1)$.
        \item We can choose local coordinates at $P$ such that $X$ is given by 
        $$
        \left(\begin{cases}
        xu+h_{+}y+g_{>1}(x,z,w) = 0 \\
        u - (x^{k-1}y+z) = 0        
        \end{cases} \right) \subseteq \mathbb{A}^{5}/\frac{1}{n}(-1,b,1,0,1),
        $$
        or
        $$
        \left(\begin{cases}
        xu+h_{+}y+g_{>1}(x,z,w) = 0 \\
        u - (x^{k-1}y+z) = 0        
        \end{cases} \right) \subseteq \mathbb{A}^{5}/\frac{1}{n}(-1,0,1,b,1)
        $$
        where $C$ is the $x$-axis and $f$ is the weighted blow-up of weight $(0,1,1,1,m')$.
    \end{enumerate}
\end{theorem}

\section{The degeneration associated to threefold divisorial contractions}

    To prove the main theorems, we want to imitate the proof of \cite[Theorem 6.8]{Self_LG_Model}, to apply the degeneration formula to the degeneration to the normal cone. However, since $X$ is not smooth, the simple normal crossing condition might not hold.
    
\begin{example}
    Consider the surface 
    $$
    X = (x^{k}y - z^{k}=0) \subseteq \mathbb{A}^3.
    $$
    Consider the blow-up $W$ of $X \times \mathbb{A}^1$ at $\{ P \} \times \{ 0 \}$. On the $x$-chart, the equation becomes
    $$
    (x'y'-z^{\prime k}=0) \in \mathbb{A}^{4}
    $$
    where the central fibre $W_{0}$ is given by $x't'=0$. Consider the embedding
    \begin{align*}
        \mathbb{A}^{2}/\frac{1}{k}(1,-1) & \rightarrow \mathbb{A}^{3} \\
        (u,v) & \mapsto (u^{k},v^{k},uv).
    \end{align*} 
    Hence $W_{0}$ is locally of the form
    $$
    \frac{\mathrm{Spec}\mathbb{C}[[u,v,t]]}{(u^{k}t)} / \frac{1}{k}(1,-1,0)
    $$
    which is not a transversal intersection of smooth Deligne-Mumford stacks.
\end{example}
    
    The above example indicates that we need to construct the family in a different way. The main results of this section are the following constructions:

\begin{theorem}\label{thm: degeneration to weighted normal cone}
    Let $g: Y \rightarrow X$ be a divisorial contraction under the assumption of \cref{main theorem 1 introduction}. Then there exists a $\mathbb{Q}$-Gorenstein projective morphism $\mathcal{W} \rightarrow C \ni 0$ such that:
    \begin{enumerate}
        \item The general fibre $\mathcal{W}_{t}$ is a $\mathbb{Q}$-smoothing of $X$;
        \item The special fibre $\mathcal{W}_{0} = Y' \cup_{E'} Y''$ is a transversal intersection between smooth Deligne-Mumford stacks.
        \item $(Y',E')$ is a $\mathbb{Q}$-smoothing of $(Y,E)$.
        \item The total space $\mathcal{W}$ has quotient singularities.
    \end{enumerate}
\end{theorem}

\begin{theorem}\label{thm: degeneration to weighted normal cone 2}
    Let $g: Y \rightarrow X$ be a divisorial contraction under the assumption of \cref{main theorem 2 introduction}. Then there exists a $\mathbb{Q}$-Gorenstein projective morphism $\mathcal{W} \rightarrow B \ni 0$ such that:
    \begin{enumerate}
        \item The general fibre $\mathcal{W}_{t}$ is a $\mathbb{Q}$-smoothing of $X$;
        \item The special fibre $\mathcal{W}_{0} = Y' \cup_{E'} Y''$ is a transversal intersection between (disjoint unions of) smooth Deligne-Mumford stacks.
        \item $(Y',E')$ contains a connected component which is a $\mathbb{Q}$-smoothing of $(Y,E)$.
        \item The total space $\mathcal{W}$ has quotient singularities.
    \end{enumerate}
\end{theorem}

    The proof of the above theorems consists of two parts:
\begin{itemize}
    \item First, we construct the family locally on $X$. The main ingredient in this part is the local classification of threefold terminal divisorial contractions, and the $\mathbb{Q}$-smoothing of a threefold terminal germ $P \in X$. The family is constructed explicitly using weighted blow-ups. This is done in \cref{section: local degeneration over a point,section: local degeneration over a curve}.
    \item Next, we show that the local construction given in the previous part can be globalized. The main ingredient in this part is the $\mathbb{Q}$-smoothing of the terminal Fano threefold $Y$ together with the exceptional divisor $E$. More explicitly, the local construction in the previous step shows the existence of the $\mathbb{Q}$-smoothing of the pair $(Y,E)$ locally around the exceptional divisor $E$, and we want to prove the existence of the $\mathbb{Q}$-smoothing of the pair $(Y,E)$ globally. The arguments are similar to the proof of the existence of $\mathbb{Q}$-smoothing for terminal Fano threefolds in \cite{Sano15}. This is done in \cref{section: global degeneration}.
\end{itemize}
    
\subsection{The local degeneration associated to threefold divisorial contractions to points}\label{section: local degeneration over a point}

    In this subsection, we prove the local version of \cref{thm: degeneration to weighted normal cone}.

\begin{proposition}[Local degeneration of threefold divisorial contractions to points]\label{prop: local degeneration to weighted normal cone}
    Let $Y \supseteq E \rightarrow X \ni P$ be a divisorial contraction between $\mathbb{Q}$-factorial terminal threefolds, which contracts the exceptional divisor $E$ to the point $P$. Assume that $P \in X$ is a smooth point, or a terminal quotient singularity, or a terminal singularity of type cA/n. Then there exists a flat morphism $\mathcal{W} \rightarrow C \ni 0$ such that:
    \begin{enumerate}
        \item The general fibre $\mathcal{W}_{t}$ is a $\mathbb{Q}$-smoothing of $P \in X$;
        \item The special fibre $\mathcal{W}_{0} = Y' \cup_{E'} Y''$ is a transversal intersection between smooth Deligne-Mumford stacks.
        \item $(Y',E')$ is a $\mathbb{Q}$-smoothing of $(Y,E)$.
        \item The total space $\mathcal{W}$ has quotient singularities.
    \end{enumerate}
\end{proposition}

    We construct such a deformation case-by-case.
    
\begin{construction}[Local degeneration of threefold divisorial contractions, Case smooth]\label{construction: smooth}
    Let $\pi: Y \rightarrow X$ be a divisorial contraction between terminal threefolds, which contracts the exceptional divisor $E \subseteq Y$ to a smooth point $P \in X$. Then by \cref{Classification of 3-fold divisorial contractions to smooth points}, we can set $P = (0,0,0) \in X = \mathbb{A}^{3}$ and $\pi$ is the weighted blow-up of weight $(1,a,b)$ for some coprime positive integers $a$ and $b$. Consider the family $X \times \mathbb{A}^1$ and the weighted blow-up $\mathcal{W}$ of $X \times \mathbb{A}^1$ at the origin with weight $(1,a,b,1)$. Then:
    \begin{enumerate}
        \item The general fibre $\mathcal{W}_{t}$ is isomorphic to $X$.
        \item On the $x$-chart, the $y$-chart and the $z$-chart, the special fibre $\mathcal{W}_{0}$ is given by
        \begin{align*}
        (xt=0) &\subseteq \mathbb{A}^4, \\
        (yt=0) &\subseteq \mathbb{A}^4/\frac{1}{a}(1,-1,b,1), \\
        (zt=0) &\subseteq \mathbb{A}^4/\frac{1}{b}(1,a,-1,1), \\
        \end{align*}
        respectively.
        \item On the $t$-chart, the special fibre $\mathcal{W}_{0}$ is irreducible.
    \end{enumerate}
    In particular, the special fibre $\mathcal{W}_{0}$ is a transversal intersection between smooth DM stacks.
\end{construction}

\begin{construction}[Local degeneration of threefold divisorial contractions, Case quotient]\label{construction: quotient}
    Let $\pi: Y \rightarrow X$ be a divisorial contraction between terminal threefolds, which contracts the exceptional divisor $E \subseteq Y$ to an isolated terminal quotient singularity $P \in X$. Then by \cref{Classification of 3-fold divisorial contractions to terminal quotient points}, we can set $P = (0,0,0) \in X = \mathbb{A}^{3}/\frac{1}{n}(s,-s,1)$ and $\pi$ is the weighted blow-up of weight $\frac{1}{n}(s,n-s,1)$. Consider the family $X \times \mathbb{A}^1$ and the weighted blow-up $\mathcal{W}$ of $X \times \mathbb{A}^1$ at the origin with weight $\frac{1}{n}(s,n-s,1,n)$. Then:
    \begin{enumerate}
        \item The general fibre $\mathcal{W}_{t}$ is isomorphic to $X$.
        \item On the $x$-chart, the $y$-chart and the $z$-chart, the special fibre $\mathcal{W}_{0}$ is given by
        \begin{align*}
        (xt=0) &\subseteq \mathbb{A}^4/\frac{1}{s}(-n,n-s,1,n), \\
        (yt=0) &\subseteq \mathbb{A}^4/\frac{1}{n-s}(s,-n,1,n), \\
        (zt=0) &\subseteq \mathbb{A}^4, \\
        \end{align*}
        respectively.
        \item On the $t$-chart, the special fibre $\mathcal{W}_{0}$ is irreducible.
    \end{enumerate}
    In particular, the special fibre $\mathcal{W}_{0}$ is a transversal intersection between smooth DM stacks.
\end{construction}

\begin{construction}[Local degeneration of threefold divisorial contractions, Case cA/n]\label{construction: cA/n}
    Let $\pi: Y \rightarrow X$ be a divisorial contraction between terminal threefolds, which contracts the exceptional divisor $E \subseteq Y$ to a point $P \in X$ of type cA/n. Then by \cref{Classification of 3-fold divisorial contractions to cA/n points}, we can set $P = (0,0,0,0) \in X = (xy+g(z^{n},w)=0) \subseteq \mathbb{A}^{4}/\frac{1}{n}(1,-1,b,0)$, where $(b,n)=1$, and $\pi$ is the weighted blow-up of weight $\frac{1}{n}(w_{1},w_{2},a,n)$. We construct the family in two steps.
    
    \textbf{Step 1: Perturbation.} We perturb the polynomial $g$ such that $g$ becomes a general polynomial of weighted order $\frac{w_{1}+w_{2}}{n}$. Such perturbation induces a $\mathbb{Q}$-smoothing of the pair germ $(Y,E)$. To accomplish this, we take a general polynomial $g'(z^{n},w)$ of weighted order $\frac{w_{1}+w_{2}}{n}$ with weights $\mathrm{wt}(z,w) = (\frac{a}{n},1)$. Consider the family
    $$
    \mathcal{X} = (xy+g(z^{n},w) + sg'(z^{n},w)=0) \subseteq \mathbb{A}^{4}/\frac{1}{n}(1,-1,b,0) \times \mathbb{A}^1.
    $$
    Consider the weighted blow-up $\mathcal{Y}$ of $\mathcal{X}$ at the $s$-axis with weight $\frac{1}{n}(w_{1},w_{2},a,n,0)$. The special fibre $\mathcal{Y}_{0} \rightarrow \mathcal{X}_{0}$ is just $\pi$ and the general fibre satisfies our conditions. We replace $\pi$ by the morphism $\mathcal{Y}_{s} \rightarrow \mathcal{X}_{s}$ for general $s$.

    \textbf{Step 2: Construction of the family.} The divisorial contraction $Y \rightarrow X$ is given by the weighted blow-up of $\mathbb{A}^{4}$ with weight $\frac{1}{n}(w_{1},w_{2},a,n)$. Consider the family $\mathcal{X}'/\mathbb{A}^1$ given by 
    $$
    (xy+g(z^{n},w)+t^{\frac{w_{1}+w_{2}}{n}}=0) \subseteq \mathbb{A}^{4}/\frac{1}{n}(1,-1,b,0) \times \mathbb{A}^1.
    $$ 
    Consider the weighted blow-up $\mathcal{W}$ of $\mathcal{X}'$ at the origin with weight $\frac{1}{n}(w_{1},w_{2},a,n,n)$. Then:
    \begin{enumerate}
        \item The general fibre $\mathcal{W}_{t}$ is a $\mathbb{Q}$-smoothing of $X$.
        \item On the $x$-chart, $\mathcal{W}$ is locally of the form
        $$
        (y+\frac{g(z^{n}x^{a},wx)}{x^{\frac{w_{1}+w_{2}}{n}}}+t^{\frac{w_{1}+w_{2}}{n}}=0) \subseteq \mathbb{A}^{5}/\frac{1}{w_{1}}(-n,w_{2},a,n,n).
        $$
        The special fibre is defined by $(xt=0)$. Hence it is a transversal intersection. Similarly, we have a transversal intersection on the $y$-chart.
        \item On the $z$-chart and the $w$-chart, $\mathcal{W}$ doesn't pass through the origin.
        \item On the $t$-chart, the special fibre $\mathcal{W}_{0}$ is irreducible.
    \end{enumerate}
    In particular, the special fibre $\mathcal{W}_{0}$ is a transversal intersection between smooth DM stacks.
\end{construction}

\begin{proof}[Proof of \cref{prop: local degeneration to weighted normal cone}]
    Follows from \cref{construction: smooth}, \cref{construction: quotient} and \cref{construction: cA/n}.
\end{proof}

    Finally, we notice that the above construction fails for general terminal singularities, as shown in the following example:

\begin{example}[A divisorial contraction to a cAx/2 point]
    Let $\pi: Y \supseteq E \rightarrow X \ni P$ be a divisorial contraction to a point, such that locally around $P$, we can choose appropriate coordinates such that $P \in X$ is given by
    $$
    (x^{2}+y^{2}+z^{2m}+w^{2m}) \subseteq \mathbb{A}^{4}/\frac{1}{2}(0,1,1,1)
    $$
    for some odd number $m$, and $\pi$ is given by the weighted blow-up of weight $\frac{1}{2}(m+1,m,1,1)$. The singularity $P \in X$ is a terminal singularity of type cAx/2. Consider the family
    $$
    (x^{2}+y^{2}+z^{10}+w^{10} + t^{5}=0) \subseteq \mathbb{A}^{4}/\frac{1}{2}(0,1,1,1) \times \mathbb{A}^{1}.
    $$
    Consider the weighted blow-up $\mathcal{W}$ given by the weight $\frac{1}{2}(m+1,m,1,1,2)$. Then:
    \begin{enumerate}
        \item The general fibre $\mathcal{W}_{t}$ is a $\mathbb{Q}$-smoothing of $X$.
        \item The exceptional divisor is given by
        $$
        (y^{2}+z^{2m}+w^{2m} + t^{m}=0) \subseteq \mathbb{P}(m+1,m,1,1,2).
        $$
        It has a singular point at $[1,0,0,0,0]$ locally given by
        $$
        (y^{2}+z^{2m}+w^{2m} + t^{m}=0) \subseteq \mathbb{A}^{4}/\frac{1}{m+1}(m,1,1,2).
        $$
        For $m \geq 5$, this singularity doesn't admit a $\mathbb{Q}$-smoothing, i.e., it cannot be deformed into quotient singularities.
        \item On the $y$-chart, the $z$-chart, the $w$-chart and the $t$-chart, $\mathcal{W}$ doesn't pass through the origin.
        \item On the $x$-chart, $\mathcal{W}$ is locally of the form
        $$
        (x+y^{2}+z^{2m}+w^{2m} + t^{m} = 0) \subseteq \mathbb{A}^{5}/\frac{1}{m+1}(-2,m,1,1,2).
        $$
        The special fibre $\mathcal{W}_{0} = Y \cup_{E} Y'$ is defined by $(xt=0)$. The component given by $(t=0)$ is isomorphic to $Y$, which has a terminal quotient singularity at the origin. The component given by $(x=0)$ is the exceptional divisor, which has an isolated singularity at the origin. The intersection is not transversal at the origin, so the simple degeneration formula doesn't work.
    \end{enumerate}
\end{example}

\subsection{The local degeneration associated to threefold divisorial contractions to curves}\label{section: local degeneration over a curve}

    In this subsection, we prove the local version of \cref{thm: degeneration to weighted normal cone 2}.

\begin{proposition}[Local degeneration of threefold divisorial contractions to curves]\label{prop: local degeneration to weighted normal cone 2}
    Let $Y \supseteq E \rightarrow X \supseteq C \ni P$ be a divisorial contraction between $\mathbb{Q}$-factorial terminal threefolds, which contracts the exceptional divisor $E$ to the smooth curve $C$. Assume that $P \in X$ is a terminal singularity of type cA or cA/n. Then there exists a flat morphism $\mathcal{W} \rightarrow B \ni 0$ such that:
    \begin{enumerate}
        \item The general fibre $\mathcal{W}_{t}$ is a $\mathbb{Q}$-smoothing of $P \in X$;
        \item The special fibre $\mathcal{W}_{0} = Y' \cup_{E'} Y''$ is a transversal intersection between smooth Deligne-Mumford stacks.
        \item $(Y',E')$ is a $\mathbb{Q}$-smoothing of $(Y,E)$.
        \item The total space $\mathcal{W}$ has quotient singularities.
    \end{enumerate}
\end{proposition}

    We construct such family case-by-case.
    
\begin{construction}[Local degeneration of threefold divisorial contractions to curves, Case 1]\label{construction: curve case 1}
    Let $\pi: Y \rightarrow X$ be a divisorial contraction between terminal threefolds, which contracts the exceptional divisor $E \subseteq Y$ to the smooth curve $C$. Assume that after taking appropriate coordinates, we can set $P = (0,0,0,0) \in X = (xy+h_{+}y+g(x,z,w)=0) \subseteq \mathbb{A}^{4}$ and $\pi$ is the weighted blow-up of weight $(0,m,1,1)$. We construct the family in two steps.

    \textbf{Step 1: Perturbation.} We perturb the polynomial $h_{+}y+g(x,z,w)$ such that it becomes a general polynomial of weighted order $m$. The argument is the same as \cref{construction: cA/n}.
    
    \textbf{Step 2: Construction of the family.} Consider the family $\mathcal{X}$ given by
    $$
    (xy+h_{+}y+g(x,z,w) + t^{m} = 0) \subseteq \mathbb{A}^{5}
    $$
    and the weighted blow-up $\mathcal{W}$ of $\mathcal{X}$ at the $x$-axis with weight $(0,m,1,1,1)$. Then:
    \begin{enumerate}
        \item The general fibre $\mathcal{W}_{t}$ is a $\mathbb{Q}$-smoothing of $X$.
        \item On the $y$-chart, $\mathcal{W}$ is locally of the form
        $$
        (x+h_{+}(x,y^{m},yz,yw)+g(x,yz,yw)/y^{m} + t^{m} = 0) \subseteq \mathbb{A}^{5}/\frac{1}{m}(0,-1,1,1,1)
        $$
        which is a quotient singularity. The special fibre is defined by $(yt=0)$. Hence it is a transversal intersection.
        \item On the $z$-chart and the $w$-chart, $\mathcal{W}$ doesn't pass through the origin.
        \item On the $t$-chart, the special fibre $\mathcal{W}_{0}$ is irreducible.
    \end{enumerate}
    In particular, the special fibre $\mathcal{W}_{0}$ is a transversal intersection between smooth DM stacks.
\end{construction}

\begin{construction}[Local degeneration of threefold divisorial contractions to curves, Case 2]\label{construction: curve case 2}
    Let $\pi: Y \rightarrow X$ be a divisorial contraction between terminal threefolds, which contracts the exceptional divisor $E \subseteq Y$ to the smooth curve $C$. Assume that after taking appropriate coordinates, we can set 
    $$
    P = (0,0,0,0,0) \in X = \left(
    \begin{cases}
    xu+h_{+}y+g_{>1}(x,z,w)=0 \\
    u-(x^{k-1}y+z) = 0
    \end{cases}
    \right) \subseteq \mathbb{A}^{5}
    $$
    and $\pi$ is the weighted blow-up of weight $(0,m'-1,1,1,m')$. We construct the family in two steps.

    \textbf{Step 1: Perturbation.} We perturb the polynomial $h_{+}y+g(x,z,w)$ such that it becomes a general polynomial of weighted order $m'$. The argument is the same as \cref{construction: cA/n}.
    
    \textbf{Step 2: Construction of the family.} Consider the family $\mathcal{X}$ given by
    $$
    \left(
    \begin{cases}
    xu+h_{+}y+g_{>1}(x,z,w) + t^{m'}=0 \\
    u-(x^{k-1}y+z) = 0
    \end{cases}
    \right) \subseteq \mathbb{A}^{6}
    $$
    and the weighted blow-up $\mathcal{W}$ of $\mathcal{X}$ at the $x$-axis with weight $(0,m'-1,1,1,m',1)$. Then:
    \begin{enumerate}
        \item The general fibre $\mathcal{W}_{t}$ is a $\mathbb{Q}$-smoothing of $X$.
        \item On the $y$-chart, the $z$-chart and the $w$-chart, $\mathcal{W}$ doesn't pass through the origin.
        \item On the $u$-chart, $\mathcal{W}$ is locally of the form
        \begin{align*}
        \left(
        \begin{cases}
        x+h_{+}(x,yu^{m'-1},zu,wu)/u^{m'}+g_{>1}(x,zu,wu)/u^{m'} + t^{m'}=0 \\
        u^{m'-1}-(x^{k-1}yu^{m'-2}+z) = 0
        \end{cases}
        \right) \\ \subseteq \mathbb{A}^{6}/\frac{1}{m'}(0,-1,1,1,-1,1)
        \end{align*}
        which is a quotient singularity. The special fibre is defined by $(ut=0)$. Hence it is a transversal intersection.
        \item On the $t$-chart, the special fibre $\mathcal{W}_{0}$ is irreducible.
    \end{enumerate}
    In particular, the special fibre $\mathcal{W}_{0}$ is a transversal intersection between smooth DM stacks.
\end{construction}

\begin{construction}[Local degeneration of threefold divisorial contractions to curves, Case 3]\label{construction: curve case 3}
    Let $\pi: Y \rightarrow X$ be a divisorial contraction between terminal threefolds, which contracts the exceptional divisor $E \subseteq Y$ to the smooth curve $C$. Assume that after taking appropriate coordinates, we can set $P = (0,0,0,0) \in X = (xy+h_{+}y+g(x,z,w)=0) \subseteq \mathbb{A}^{4}/\frac{1}{r}(r-1,1,\alpha,r)$ and $\pi$ is the weighted blow-up of weight $(0,m,1,1)$. We construct the family in two steps.

    \textbf{Step 1: Perturbation.} We perturb the polynomial $h_{+}y+g(x,z,w)$ such that it becomes a general polynomial of weighted order $m$. The argument is the same as \cref{construction: cA/n}.
    
    \textbf{Step 2: Construction of the family.} Consider the family $\mathcal{X}$ given by
    $$
    (xy+h_{+}y+g(x,z,w) + t^{m} = 0) \subseteq \mathbb{A}^{4}/\frac{1}{r}(-1,1,\alpha,0) \times \mathbb{A}^1
    $$
    and the weighted blow-up $\mathcal{W}$ of $\mathcal{X}$ at the $x$-axis with weight $(0,m,1,1,1)$. Then:
    \begin{enumerate}
        \item The general fibre $\mathcal{W}_{t}$ is a $\mathbb{Q}$-smoothing of $X$.
        \item On the $y$-chart, $\mathcal{W}$ is locally of the form
        $$
        (x+h_{+}(x,y^{m},yz,yw)+g(x,yz,yw)/y^{m} + t^{m} = 0) \subseteq \mathbb{A}^{5}/\frac{1}{mr}(-m,1,m\alpha-1,-1,-1)
        $$
        which is a quotient singularity. The special fibre is defined by $(yt=0)$. Hence it is a transversal intersection.
        \item On the $z$-chart and the $w$-chart, $\mathcal{W}$ doesn't pass through the origin.
        \item On the $t$-chart, the special fibre $\mathcal{W}_{0}$ is irreducible.
    \end{enumerate}
    In particular, the special fibre $\mathcal{W}_{0}$ is a transversal intersection between smooth DM stacks.
\end{construction}

\begin{construction}[Local degeneration of threefold divisorial contractions to curves, Case 4]\label{construction: curve case 4}
    Let $\pi: Y \rightarrow X$ be a divisorial contraction between terminal threefolds, which contracts the exceptional divisor $E \subseteq Y$ to the smooth curve $C$. Assume that after taking appropriate coordinates, we can set 
    $$
    P = (0,0,0,0,0) \in X = \left(
    \begin{cases}
    xu+h_{+}y+g_{>1}(x,z,w)=0 \\
    u-(x^{k-1}y+z) = 0
    \end{cases}
    \right) \subseteq \mathbb{A}^{5}/\frac{1}{r}(-1,\alpha,1,0,1)
    $$
    or
    $$
    P = (0,0,0,0,0) \in X = \left(
    \begin{cases}
    xu+h_{+}y+g_{>1}(x,z,w)=0 \\
    u-(x^{k-1}y+z) = 0
    \end{cases}
    \right) \subseteq \mathbb{A}^{5}/\frac{1}{r}(-1,0,1,\alpha,1)
    $$
    and $\pi$ is the weighted blow-up of weight $(0,1,1,1,m')$. Without loss of generality, we only study the first case. We construct the family in two steps.

    \textbf{Step 1: Perturbation.} We perturb the polynomial $h_{+}y+g(x,z,w)$ such that it becomes a general polynomial of weighted order $m'$. The argument is the same as \cref{construction: cA/n}.
    
    \textbf{Step 2: Construction of the family.} Consider the family $\mathcal{X}$ given by
    $$
    \left(
    \begin{cases}
    xu+h_{+}y+g_{>1}(x,z,w) + t^{m'}=0 \\
    u-(x^{k-1}y+z) = 0
    \end{cases}
    \right) \subseteq \mathbb{A}^{5}/ \frac{1}{r}(-1,\alpha,1,0,1) \times \mathbb{A}^1
    $$
    and the weighted blow-up $\mathcal{W}$ of $\mathcal{X}$ at the $x$-axis with weight $(0,1,1,1,m',1)$. Then:
    \begin{enumerate}
        \item The general fibre $\mathcal{W}_{t}$ is a $\mathbb{Q}$-smoothing of $X$.
        \item On the $y$-chart, the $z$-chart and the $w$-chart, $\mathcal{W}$ doesn't pass through the origin.
        \item On the $u$-chart, $\mathcal{W}$ is locally of the form
        \begin{align*}
        \left(
        \begin{cases}
        x+h_{+}(x,yu,zu,wu)/u^{m'}+g_{>1}(x,zu,wu)/u^{m'} + t^{m'}=0 \\
        u^{m'-1}-(x^{k-1}y+z) = 0
        \end{cases}
        \right) \\ \subseteq \mathbb{A}^{6}/\frac{1}{m'r}(-m',m'\alpha - 1,1,-1,-1,1).
        \end{align*}
        The special fibre is defined by $(ut=0)$. Hence it is a transversal intersection.
        \item On the $t$-chart, the special fibre $\mathcal{W}_{0}$ is irreducible.
    \end{enumerate}
    In particular, the special fibre $\mathcal{W}_{0}$ is a transversal intersection between smooth DM stacks.
\end{construction}

\subsection{The global construction}\label{section: global degeneration}

    In this subsection, we prove \cref{thm: degeneration to weighted normal cone} and \cref{thm: degeneration to weighted normal cone 2} by showing that the above constructions can be globalized. The proof basically follows the same idea as \cite{Sano15}. The difference is that we need a global $\mathbb{Q}$-smoothing of the pair $(Y,E)$ instead of only the terminal Fano threefold $Y$.
    
\begin{proposition}\label{prop: first-order deformation of pairs}
    Let $X$ be a geometric local scheme of dimension $\geq 3$ and $D$ a normal Cartier divisor on $X$ such that 
    \begin{enumerate}
        \item $X$ is smooth at $x$;
        \item the pair $(X,D)$ is log smooth on $U := X-x$.
    \end{enumerate}
    Then there is an isomorphism
    $$
    T_{(X,D)}^{1} \simeq H^{1}(U,\Theta_{X}(-log D))
    $$
    where $T_{(X,D)}^{1}$ is the tangent space to the formal moduli space of $(X,D)$. In particular, $(X,D)$ is rigid if and only if $H^{1}(U,\Theta_{X}(-log D)) = 0$. Moreover, assume that there is a finite group $G$ acting on $X$ and $D$ is $G$-invariant, then the same holds for the $G$-equivariant deformation.
\end{proposition}

\begin{proof}
    First, we show that $T_{(X,D)}^{1} \cong \mathrm{Ext}^{1}(\Omega_{X}^{1}(log D),\mathcal{O}_{X})$. Indeed, we have $T_{(X,D)}^{1} \cong \mathrm{Ext}^{1}(L_{(X,D)}^{log},\mathcal{O}_{X})$ where $L_{(X,D)}^{log}$ is the logarithmic cotangent complex. We have a long exact sequence
    $$
    \cdots \mathrm{Hom}(A,\mathcal{O}_{X}) \rightarrow \mathrm{Ext}^{1}(\Omega_{X}^{1}(log D),\mathcal{O}_{X}) \rightarrow \mathrm{Ext}^{1}(\tau_{\geq -1}L_{(X,D)}^{log},\mathcal{O}_{X}) \rightarrow \mathrm{Ext}^{1}(A,\mathcal{O}_{X})
    $$
    given by the truncation. Since $(U,D_{U})$ is log smooth, $A$ is supported on the closed point $x$. Hence $\mathrm{Hom}(A,\mathcal{O}_{X}) = \mathrm{Ext}^{1}(A,\mathcal{O}_{X}) = 0$ and we have the desired isomorphism.

    Since $(X,D)$ is log smooth on $U$, we have $\mathrm{Ext}^{1}(\Omega_{U}^{1}(log D_{U}),\mathcal{O}_{U}) \cong H^{1}(U,\Theta_{X}(-log D))$. Hence it suffices to show that the natural morphism
    $$
    \mathrm{Ext}^{1}(\Omega_{X}^{1}(log D),\mathcal{O}_{X}) \rightarrow \mathrm{Ext}^{1}(\Omega_{U}^{1}(log D_{U}),\mathcal{O}_{U})
    $$
    is an isomorphism.
    
    Since $X$ is smooth and $D$ is normal, we have the residue short exact sequence
    $$
    0 \rightarrow \Omega_{X}^{1} \rightarrow \Omega_{X}^{1}(log D) \rightarrow \mathcal{O}_{D} \rightarrow 0.
    $$
    Applying the functor $\mathcal{H}om({-},\mathcal{O}_{X})$, we have the long exact sequence
    $$
    0 \rightarrow \Theta_{X}(-log D) \rightarrow \Theta_{X} \rightarrow \mathcal{O}_{D}(D) \rightarrow \mathcal{E}xt^{1}(\Omega_{X}^{1}(log D),\mathcal{O}_{X}) \rightarrow 0.
    $$
    Since $X$ is smooth at $x$ of dimension $\geq 3$ and $D$ is normal, we have $\mathrm{depth}_{x} \Theta_{X} \geq 3$ and $\mathrm{depth}_{x}\mathcal{O}_{D}(D) \geq 2$. Since $(X,D)$ is log smooth on $U$, the sheaf $\mathcal{E}xt^{1}(\Omega_{X}^{1}(log D),\mathcal{O}_{X})$ is a skyscraper sheaf supported on the closed point $x$. Hence we have 
    \begin{equation}\label{equation 1}
    H_{x}^{i}(X,\mathcal{E}xt^{1}(\Omega_{X}^{1}(log D),\mathcal{O}_{X})) = H^{i}(X,\mathcal{E}xt^{1}(\Omega_{X}^{1}(log D),\mathcal{O}_{X})) = 0
    \end{equation}
    for $i > 0$. Combining the above results, we have
    \begin{equation}\label{equation 2}
        H_{x}^{1}(X,\Theta_{X}(-log D)) = 0
    \end{equation}
    and
    \begin{equation}\label{equation 3}
    H_{x}^{2}(X,\Theta_{X}(-log D)) \cong H^{0}(X,\mathcal{E}xt^{1}(\Omega_{X}^{1}(log D),\mathcal{O}_{X})).
    \end{equation}

    We have the long exact sequence
    $$
    H_{x}^{1}(X,K) \rightarrow \mathrm{Ext}^{1}(\Omega_{X}^{1}(log D),\mathcal{O}_{X}) \rightarrow \mathrm{Ext}^{1}(\Omega_{U}^{1}(log D_{U}),\mathcal{O}_{U}) \rightarrow H_{x}^{2}(X,K)
    $$
    where $K := \mathcal{RH}om(\Omega_{X}^{1}(log D),\mathcal{O}_{X})$. The local cohomology can be computed by the spectral sequence
    $$
    E^{pq}_{2} = H^{p}_{x}(X,H^{q}(K)) \Rightarrow H^{p+q}_{x}(K).
    $$
    We have $H^{0}(K) = \Theta_{X}(-log D)$, $H^{1}(K) = \mathcal{E}xt^{1}(\Omega_{X}^{1}(log D),\mathcal{O}_{X})$ and $H^{i}(K) = 0$ for $i \neq 0,1$. Hence the associated 7-term exact sequence is 
    \begin{align*}
    0 \rightarrow H^{1}_{x}(X,\Theta_{X}(-log D)) \rightarrow H_{x}^{1}(X,K) \rightarrow H^{0}_{x}(X,\mathcal{E}xt^{1}(\Omega_{X}^{1}(log D),\mathcal{O}_{X})) \\ \rightarrow H_{x}^{2}(X,\Theta_{X}(-log D)) \rightarrow H_{x}^{2}(X,K) \rightarrow H^{1}_{x}(X,\mathcal{E}xt^{1}(\Omega_{X}^{1}(log D),\mathcal{O}_{X})) \\ \rightarrow H^{3}_{x}(X,\Theta_{X}(-log D)).
    \end{align*}
    By (\ref{equation 1}), (\ref{equation 2}) and (\ref{equation 3}) we conclude that $H_{x}^{1}(X,K) = H_{x}^{2}(X,K) = 0$. Hence we have the isomorphism
    $$
    \mathrm{Ext}^{1}(\Omega_{X}^{1}(log D),\mathcal{O}_{X}) \cong \mathrm{Ext}^{1}(\Omega_{U}^{1}(log D_{U}),\mathcal{O}_{U}).
    $$
\end{proof}

    Now we need a special pair version of the main result in \cite{NamikawaSteenbrink1995}.

\begin{proposition}\label{prop: coboundary map of deformation of pairs}
    Let $(x \in D \subset X) \rightarrow (z \in Z)$ be a divisorial contraction germ in dimension 3 such that 
    \begin{enumerate}
        \item $x \in X$ is smooth and $z \in Z$ is a Gorenstein Du Bois singularity;
        \item the exceptional divisor $D$ is Du Val at $x \in X$;
        \item the pair $(X,D)$ is log smooth on $U := X-x$.
    \end{enumerate}
    Let $Y \rightarrow X$ be a log resolution of $(X,D)$ centered at $x$ and $E$ be the exceptional divisor. Denote by $\widetilde{D}$ the strict birational transformation of $D$ on $Y$. We have a natural map $\tau : H^{1}(U,\Omega_{X}^{2}(log D)(-D)) \rightarrow H^{2}_{E}(Y,\Omega_{Y}^{2}(log \widetilde{D})(- \widetilde{D}))$. Suppose that $\tau$ is the zero map. Then $(X,D)$ is rigid.
\end{proposition}

\begin{proof}
    By \cref{prop: first-order deformation of pairs}, $(X,D)$ is rigid if and only if $H^{1}(U,\Theta_{X}(-log D)) = 0$. We use the isomorphism $\omega_{X} \cong \mathcal{O}_{X}$ to identify the sheaf $\Theta_{X}(-log D)$ with $\Omega^{2}_{X}(log D)(-D)$. The map $\tau$ can be factorized through $H^{2}_{E}(Y,\Omega_{Y}^{2}(log (\widetilde{D} + E))(-\widetilde{D}-E))$. By the vanishing theorem of Guill\'en, Navarro Aznar and Puerta, we have $H^{2}(Y,\Omega_{Y}^{2}(log (\widetilde{D} + E))(-\widetilde{D}-E)) = 0$. Hence the map $H^{1}(U,\Omega_{X}^{2}(log D)(-D)) \rightarrow H^{2}_{E}(Y,\Omega_{Y}^{2}(log (\widetilde{D} + E))(-\widetilde{D}-E))$ is surjective. Define
    \begin{align*}
    \omega_{E}^{p}(log \widetilde{D})(- \widetilde{D}) & := \Omega_{E}^{p}(log \widetilde{D})(-\widetilde{D}) \text{ mod torsion } \\
    & \cong \Omega_{Y}^{p}(log \widetilde{D})(-\widetilde{D})/\Omega_{Y}^{p}(log (\widetilde{D} + E))(-\widetilde{D}-E).
    \end{align*}
    Then we have the exact sequence
    $$
    H^{1}(E,\omega_{E}^{2}(log \widetilde{D})(- \widetilde{D})) \xrightarrow{\alpha} H^{2}_{E}(Y,\Omega_{Y}^{2}(log (\widetilde{D} + E))(-\widetilde{D}-E)) \rightarrow H^{2}_{E}(Y,\Omega_{Y}^{2}(log \widetilde{D})(-\widetilde{D}))
    $$
    where $\alpha$ can be factored through
    $$
    \alpha': H^{1}(E,\omega_{E}^{2}(log \widetilde{D})(- \widetilde{D})) \rightarrow H^{1}(E, \Omega_{Y}^{2}(log (\widetilde{D} + E))(-\widetilde{D})\otimes \mathcal{O}_{E})).
    $$
    The morphism $\alpha'$ can be interpreted as $\mathrm{Gr}_{F}^{2}H^{3}(E,\widetilde{D}|_{E},\mathbb{C}) \rightarrow \mathrm{Gr}_{F}^{2}H^{4}_{x}(X,D,\mathbb{C})$. Since $\widetilde{D} + E$ is a simple normal crossing divisor, we have $\omega_{E}^{2}(log \widetilde{D})(- \widetilde{D}) \cong \omega_{E}^{2}$ and $\mathrm{Gr}_{F}^{2}H^{3}(E,\widetilde{D}|_{E},\mathbb{C}) \cong \mathrm{Gr}_{F}^{2}H^{3}(E,\mathbb{C})$. By semi-purity, the combination 
    $$
    \mathrm{Gr}_{F}^{2}H^{3}(E,\mathbb{C}) \xrightarrow{\alpha'} \mathrm{Gr}_{F}^{2}H^{4}_{x}(X,D,\mathbb{C}) \rightarrow \mathrm{Gr}_{F}^{2}H^{4}_{x}(X,\mathbb{C})
    $$
    is the zero morphism. Consider the natural long exact sequence
    $$
    \cdots \rightarrow H^{3}_{x}(D,\mathbb{C}) \rightarrow H^{4}_{x}(X,D,\mathbb{C}) \rightarrow H^{4}_{x}(X,\mathbb{C}) \rightarrow \cdots
    $$
    Since $x \in D$ is a Du Val singularity, it is locally of the form $\mathbb{C}^{2}/G$ for some finite group $G \leq \mathrm{SL}(2,\mathbb{C})$. Hence we have $H^{3}_{x}(D,\mathbb{C}) \cong \widetilde{H}^{2}(S^{3}/G,\mathbb{C}) = 0$. Hence the morphism $H^{4}_{x}(X,D,\mathbb{C}) \rightarrow H^{4}_{x}(X,\mathbb{C})$ is injective. Hence $\alpha'$ is the zero morphism. Hence $\alpha$ is the zero map and the morphism $H^{2}_{E}(Y,\Omega_{Y}^{2}(log (\widetilde{D} + E))(-\widetilde{D}-E)) \rightarrow H^{2}_{E}(Y,\Omega_{Y}^{2}(log \widetilde{D})(-\widetilde{D}))$ is injective. Combining the above results we have
    $$
    \mathrm{im}\tau \cong H^{2}_{E}(Y,\Omega_{Y}^{2}(log (\widetilde{D} + E))(-\widetilde{D}-E)).
    $$
    Suppose that $\tau$ is the zero map. Then $H^{2}_{E}(Y,\Omega_{Y}^{2}(log (\widetilde{D} + E))(-\widetilde{D}-E)) = 0$. We have the short exact sequence
    \begin{align*}
    H^{1}(Y,\Omega_{Y}^{2}(log (\widetilde{D} + E))(-\widetilde{D}-E)) \rightarrow H^{1}(U,\Omega^{2}_{X}(log D)(-D)) \\ \rightarrow H^{2}_{E}(Y,\Omega_{Y}^{2}(log (\widetilde{D} + E))(-\widetilde{D}-E)).
    \end{align*}
    We have $H^{1}(Y,\Omega_{Y}^{2}(log (\widetilde{D} + E))(-\widetilde{D}-E)) = 0$ by \cite{NamikawaSteenbrink1995}. Hence we conclude that $H^{1}(U,\Omega^{2}_{X}(log D)(-D)) = 0$ and $(X,D)$ is rigid.
    \end{proof}
    
\begin{remark}
    \cref{prop: coboundary map of deformation of pairs} can be proved in a more general situation, where $x \in D$ is an arbitrary isolated singularity, by an explicit calculation on $H^{4}_{x}(X,D,\mathbb{C})$.
\end{remark}

\begin{proof}[Proof of \cref{thm: degeneration to weighted normal cone}]
    It suffices to show that the family $\mathcal{X}$ and $\mathcal{X}'$ constructed in \cref{construction: smooth}, \cref{construction: quotient} and \cref{construction: cA/n} can be globalized. When $P \in X$ is a smooth point or a terminal quotient singularity, $\mathcal{W}$ can be constructed by blowing up the trivial family. Hence we can assume that $P \in X$ is a terminal singularity of type cA/n.

    \textbf{Step 1: Simple reduction.} First, we show that we can reduce to the case where $Y$ has only terminal quotient singularities. Indeed, by assumption $Y$ is a Fano threefold with ordinary terminal singularities. Hence by \cref{thm: Q-smoothing of Fano threefolds with ordinary terminal singularities} there exists a $\mathbb{Q}$-smoothing $\mathcal{Y} \rightarrow S \ni 0$ of $Y$. Since $Y$ is $\mathbb{Q}$-factorial, by \cref{thm: deformation of extremal contractions} the divisorial contraction $g: Y \rightarrow X$ at the special fibre can be deformed to a divisorial contraction $g_{s}: Y_{s} \rightarrow X_{s}$ at the general fibre. Applying the semi-continuity of fibre dimensions on the center of the divisorial contraction, we conclude that the divisorial contraction at the general fibre also contracts the exceptional divisor $E_{s}$ to a point $P_{s}$. By assumption the point $P \in X$ is a smooth point or a quotient point or a terminal singularity of type cA/n. After a deformation, the point $P_{s} \in X_{s}$ is still one of the above types. Hence we can replace $g$ by $g_{s}$ and assume that $Y$ has only terminal quotient singularities. 
    
    \textbf{Step 2: Globalizing the perturbation.} Let $\pi|_{U}: U \supseteq E \rightarrow \pi(U) \ni P$ be the germ of the divisorial contraction. We want to find a perturbation of $X$ such that the restriction on $P \in \pi(U)$ satisfies the condition of \cref{construction: cA/n}, Step 1. Indeed, recall that the local equation of $P \in \pi(U)$ is given by
    $$
    (xy + g(z^{n},w)=0) \subseteq \mathbb{A}^{4}/\frac{1}{n}(1,-1,b,0)
    $$
    for some $n \geq 2$. Moreover, the polynomial $g(z^{n},w)$ contains a monomial of the form $w^{l}$ for some integer $l \geq \frac{w_{1}+w_{2}}{n}$. Indeed, if $g$ doesn't have such a term, then the singular locus is locally given by the equation $x=y=0$, which is not isolated. Since $Y$ has only terminal quotient singularities, it has no isolated cDV singularities. Hence it suffices to make the perturbation such that the term $w^{\frac{w_{1}+w_{2}}{n}}$ has a non-zero coefficient in $g$. If $w^{\frac{w_{1}+w_{2}}{n}}$ already has non-zero coefficient in $g$ then there is nothing to prove. Hence we can assume that $g$ doesn't contain the monomial term $w^{\frac{w_{1}+w_{2}}{n}}$. Let $l \geq \frac{w_{1}+w_{2}}{n} + 1$ be the smallest integer such that $w^{l}$ has a non-zero coefficient in $g(z^{n},w)$. On the $w$-chart of the weighted blow-up given in \cref{construction: cA/n}, the equation of $Y$ is
    $$
    (xy + g(z^{n}w^{a},w)/w^{\frac{w_{1}+w_{2}}{n}} = 0) \subseteq \mathbb{A}^{4}/\frac{1}{n}(w_{1},w_{2},a,0).
    $$
    If $l > \frac{w_{1}+w_{2}}{n} + 1$, then by \cref{thm: Classification of threefold terminal singularities II}, $Y$ has a non-quotient terminal singularity of type cA/n, a contradiction to our assumption. Hence $g$ must contain the monomial term $w^{\frac{w_{1}+w_{2}}{n}+1}$. Then $Y$ has a terminal quotient singularity at the origin $Q$. The exceptional divisor $E$ is locally given by the equation
    $$
    (xy + z^{\frac{w_{1}+w_{2}}{a}} = 0) \subseteq \mathbb{A}^{3}/\frac{1}{n}(w_{1},w_{2},a)
    $$
    which also has a singularity at $Q$. Consider the perturbation of $P \in \pi(U)$ given by
    $$
    (xy + g(z^{n},w) + tw^{\frac{w_{1}+w_{2}}{n}} = 0 ) \subseteq \mathbb{A}^{4}/\frac{1}{n}(1,-1,b,0).
    $$
    This perturbation induces a deformation of $(U,E)$. The associated first-order deformation $\eta_{U} \in \mathrm{Ext}^{1}(\Omega_{U}(log E),\mathcal{O}_{U})$ is non-trivial. We show that $\eta_{U}$ can be extended to a first-order deformation of $(Y,E)$ following \cite{Sano15}. 
    
    First, we construct the index-one cover of $Y$. More explicitly, we take a sufficiently large positive integer $m$ such that $-mK_{Y}$ is very ample and there is an element $D_{m} \in |-mK_{Y}|$ which doesn't contain any singular point of $Y$. Let 
    $$
    \pi' : Z := \mathrm{Spec} \bigoplus\limits_{i=0}^{m-1}\mathcal{O}_{Y}(iK_{Y}) \rightarrow Y
    $$
    be the cyclic cover determined by $D_{m}$. Since $Y$ has only terminal quotient singularities, $Z$ is smooth and $F := \pi^{\prime -1}E$ has Du Val singularities at $\pi^{\prime -1}(Q)$. There exists a good $\mathbb{Z}_{m}$-equivariant log resolution $\nu: \widetilde{Z} \rightarrow Z$ of the pair $(Z,F)$ which induces an isomorphism $\nu^{-1}(Z\backslash \pi^{\prime -1}(Q))$ and a birational morphism $\mu: \widetilde{Y} := \widetilde{Z}/\mathbb{Z}_{m} \rightarrow Y$. This induces the following Cartesian diagram:
    $$
    \begin{tikzcd}
    \widetilde{Z} \arrow[r,"\widetilde{\pi}'"] \arrow[d,"\nu"] & \widetilde{Y} \arrow[d,"\mu"] \\
    Z \arrow[r,"\pi'"] & Y
    \end{tikzcd}
    $$
    Let $V = \pi^{\prime -1}(U)$, $\widetilde{V} = \nu^{-1}(V)$ and $\widetilde{U} = \widetilde{V}/\mathbb{Z}_{m}$. Also put $F' = \mathrm{Exc}(\nu)$ and $E' = \mathrm{Exc}(\mu)$. Let $Q' \in \pi^{\prime -1}(Q)$. There is an element $\eta_{F} \in T^{1}_{V,F,Q'}$ given by
    \begin{align*}
        V_{t}:(xy + g(z^{n}w^{a},w)/w^{\frac{w_{1}+w_{2}}{n}} + t = 0) & \subseteq \mathbb{A}^{4} \\
        F_{t}:(w=0) & \subseteq V_{t}.
    \end{align*}
    
    Note that $T^{1}_{(V,F),Q'}$ is generated by $\eta_{F}$ as an $\mathcal{O}_{V,Q'}$-module. Let $\mathcal{F}^{(0)}$ be the $\mathbb{Z}_{m}$-invariant part of $\widetilde{\pi}'_{*}(\Omega_{\widetilde{Z}}^{2}(log (F+F'))(-F-F'-\nu^{*}K_{Z}))$. We have the commutative diagram
    $$
    \begin{tikzcd}
    H^{1}(Y_{sm},\Omega_{Y_{sm}}^{2}(log E)(-E) \otimes \omega_{Y_{sm}}^{-1}) \arrow[r,"\psi"] \arrow[d] & H_{E'}^{2}(\widetilde{Y},\mathcal{F}^{(0)}) \arrow[d,"\cong"] \arrow[r] & H^{2}(\widetilde{Y},\mathcal{F}^{(0)}) \\
    H^{1}(U'_{Q},\Omega_{U'_{Q}}^{2}(log E)(-E) \otimes \omega_{U'_{Q}}^{-1}) \arrow[r,"\phi"] & H_{E'}^{2}(\widetilde{U_{Q}},\mathcal{F}^{(0)}) & 
    \end{tikzcd}
    $$
    where $Y_{sm}$ is the smooth locus of $Y$, $U_{Q}$ is a Stein neighborhood of $Q$, $U'_{Q} = U_{Q} - \{ Q \}$ is the punctured neighborhood of $Q$. By \cite{Sano15}, p.16, we know that $H^{2}(\widetilde{Y},\mathcal{F}^{(0)}) = 0$. Hence $\psi$ is surjective. Hence we can find an element $\widetilde{\eta} \in H^{1}(Y_{sm},\Omega_{Y_{sm}}^{2}(log E)(-E) \otimes \omega_{Y_{sm}}^{-1})$ such that $\psi(\widetilde{\eta}) = \phi(\eta_{F})$. By \cref{thm: Unobstructed deformations of terminal Fano threefolds}, deformations of the terminal Fano threefold $Y$ are unobstructed, so the first-order deformation given by $\widetilde{\eta}$ integrates to a $\mathbb{Q}$-Gorenstein deformation of $Y$. By \cref{thm: deformation of extremal contractions}, the divisorial contraction deforms to a divisorial contraction, so the exceptional divisor $E$ deforms together with $Y$. Hence we obtained the desired deformation.
    
    \textbf{Step 3: Globalizing the $\mathbb{Q}$-smoothing.} The local $\mathbb{Q}$-smoothing
    $$
    (xy+g(z^{n},w)+t=0) \subseteq \mathbb{A}^{4}/\frac{1}{n}(1,-1,b,0) \times \mathbb{A}^1
    $$ 
    given in \cref{construction: cA/n} can be globalized by \cite[p.14]{Sano15}. After a finite base change of degree $\frac{w_{1}+w_{2}}{n}$, we will obtain the deformation with the desired degree.
\end{proof}

\begin{proof}[Proof of \cref{thm: degeneration to weighted normal cone 2}]
    Similar to the proof of \cref{thm: degeneration to weighted normal cone}, it suffices to show that the local family constructed can be globalized.

    \textbf{Step 1: Simple reduction.} First, we show that we can reduce to the case where $Y$ has only terminal quotient singularities. This is the same as the proof of \cref{thm: degeneration to weighted normal cone}.
    
    \textbf{Step 2: Globalizing the perturbation.} Let $\pi|_{U}: U \supseteq E \rightarrow \pi(U) \ni P$ be the germ of the divisorial contraction. We want to find a perturbation of $X$ such that the restriction on $P \in \pi(U)$ satisfies the condition of \cref{construction: curve case 1}, \cref{construction: curve case 2}, \cref{construction: curve case 3}, and \cref{construction: curve case 4}. This is the same as the proof of \cref{thm: degeneration to weighted normal cone}.
    
    \textbf{Step 3: Globalizing the ($\mathbb{Q}$-)smoothing.} Let $P_{1},\cdots,P_{l}$ be the singular points on $C$. By our assumption, each $P_{i}$ is a terminal singularity of type cA or cA/n. At each $P_{i}$, we have the local ($\mathbb{Q}$-)smoothing
    $$
    (xy+g(z^{n_{i}},w)+t=0) \subseteq \mathbb{A}^{4}/\frac{1}{n_{i}}(1,-1,b_{i},0) \times \mathbb{A}^1
    $$
    given in \cref{construction: cA/n}. These local ($\mathbb{Q}$-)smoothings can be globalized by \cite{Sano15}. Let $\mathcal{X} \rightarrow B$ be the ($\mathbb{Q}$-)smoothing. At each $P_{i}$, we can construct a local degeneration as in \cref{prop: local degeneration to weighted normal cone 2} after a base change of degree $m_{i}$. It remains to glue up the local degeneration to a global one. To construct a global family, we take a base change of degree $M = \mathrm{lcm}(m_{1},m_{2},\cdots,m_{l})$. On the smooth locus of $\mathcal{X}$ near $C$, we can write the coordinate as $\mathbb{A}^{4}_{x,y,z,t}$, where $C$ is defined by $y=z=t=0$ and $t$ is the parameter on $B$. We construct a common weighted blow-up of weights $(0,\frac{M}{m_{i}},\frac{M}{m_{i}},1)$ for all $i$. In the toric picture, the common weighted blow-up is constructed by a subdivision of the standard simplex by the rays generated by $(0,\frac{M}{m_{i}},\frac{M}{m_{i}},1)$ for all $i$. Notice that since all the primitive generators of these rays lie in the line $\{(0,r,r,1) \mid r \in \mathbb{R}\}$, the resulting fan is independent of the order of the subdivisions. Hence such a common weighted blow-up is canonical. At each singular point $P_{i_{0}}$ of Case 1, these extractions are extended to the weighted blow-ups of weights $(0,\frac{m_{i_{0}}M}{m_{i}},\frac{M}{m_{i}},\frac{M}{m_{i}},1)$ of the ambient space. Since all these weight vectors lie in the line $\{(0,m_{i_{0}}r,r,r,1)\mid r \in \mathbb{R}\}$, the resulting fan is independent of the order of the subdivisions. Therefore, without affecting the projectivity, we can first take the weighted blow-up of weight $(0,M,\frac{M}{m_{i_{0}}},\frac{M}{m_{i_{0}}},1)$, and then take the weighted blow-up of other weights. Then by \cref{construction: curve case 1}, the central fibre is a transversal intersection of smooth DM stacks, and the total space has quotient singularities. Hence we can take further weighted blow-ups to extract other exceptional divisors. Similar constructions can be made for singular points of Case 2,3, and 4. We can then glue these local extractions together and obtain $\mathcal{W}$. The central fibre $\mathcal{W}_{0}$ may have more than 2 irreducible components, but we have a decomposition $\mathcal{W}_{0} = Y' \cup Y''$ where $Y'$ and $Y''$ are (disjoint unions of) smooth DM stacks and intersect transversally. Indeed, on the smooth locus of $\mathcal{X}$ near $C$, we can locally build a toric fan given by the rays $(1,0,0,0),(0,1,0,0),(0,0,1,0),(0,0,0,1)$. The strict birational transformation of $\mathcal{X}_{0}$ is given by the ray $(0,0,0,1)$, and the exceptional divisors is given by the rays $(0,\frac{M}{m_{i}},\frac{M}{m_{i}},1)$. If we sort the numbers $\frac{M}{m_{i}}$, then two exceptional divisors $E_{i}$ and $E_{j}$ intersect if and only if $\frac{M}{m_{i}}$ and $\frac{M}{m_{j}}$ are adjacent. Hence we can divide the central fibre into two parts $Y'$ and $Y''$, which are disjoint unions of smooth DM stacks over the smooth locus of $\mathcal{X}$. On the other side, the central fibre is a transversal intersection by our construction. Hence $Y'$ and $Y''$ are also transversal over the singular locus.
\end{proof}

\section{Proof of the main theorems}
    In this section, we prove \cref{main theorem 1 introduction} and \cref{main theorem 2 introduction}. First, we show that we can take the limit of the parametrized quantum period.

\begin{lemma}\label{lemma: limit of quantum period}
    Let $Y$ be a Fano variety of terminal quotient singularities, and $E$ be a Weil divisor on $Y$ such that the preimage $\widetilde{E}$ of $E$ on the Gorenstein index-one cover $\widetilde{Y}$ of $Y$ is smooth. Then for every curve class $\widetilde{\gamma}$ on $Y$ such that $\langle \tau_{-K_{Y}\cdot\widetilde{\gamma} -2}\mathbf{1} \rangle_{\widetilde{\gamma}}^{Y} \neq 0$, we have $E \cdot \widetilde{\gamma} \geq 0$. In particular, the limit $\lim\limits_{r \rightarrow +\infty} \hat{G}_{Y,rE}(t)$ exists.
\end{lemma}

\begin{proof}
    Since $(Y,E)$ is an orbifold pair, we can construct the degeneration of $Y$ to the orbifold normal cone as follows: 
    \begin{enumerate}[align=left,label=(\arabic*)]
        \item First, take the Gorenstein index-one cover $(\widetilde{Y},\widetilde{E})$ of $(Y,E)$ and denote by $G$ the associated cyclic quotient group;
        \item Next, take the degeneration of $Y$ to the normal cone of $E$. More explicitly, the family $\widetilde{\mathcal{Y}}/\mathbb{A}^1$ is the blow-up of $\widetilde{Y} \times \mathbb{A}^1$ at $\widetilde{E} \times \{ 0 \}$;
        \item Finally, take the cyclic quotient $\mathcal{Y}/\mathbb{A}^1$ of $\widetilde{\mathcal{Y}}/\mathbb{A}^1$ by the naturally induced action of $G$.
    \end{enumerate}
    By construction, the total space $\widetilde{\mathcal{Y}}$ is an orbifold and the central fibre $\widetilde{\mathcal{Y}}_{0}$ is a transversal intersection between orbifolds. We apply \cref{theorem: orbifold degeneration formula} to the family $\widetilde{\mathcal{Y}}/\mathbb{A}^1$. The point class can be chosen to be outside of $E$, so all insertions remain in $Y$. Let $\widetilde{\gamma} = \widetilde{\gamma}_{1} + \widetilde{\gamma}_{2}$ be a splitting of the curve class $\widetilde{\gamma}$, $(\delta_{1},\cdots,\delta_{\rho})$ be the contact data and $\mu_{i} = (E \cdot \widetilde{\gamma}_{1})_{\delta_{i}}$ be the associated intersection multiplicity. Then the relative virtual dimension on the $Y$-side is given by
    $$
    d_{\Gamma_{1}} = -K_{Y} \cdot \widetilde{\gamma}_{1} + \rho - \mu - \sum\limits_{i=1}^{\rho} \mathrm{Age}_{E}(\delta_{i}) = d_{\widetilde{\gamma}} + K_{Y} \cdot \widetilde{\gamma}_{2} + \rho -\mu - \sum\limits_{i=1}^{\rho} \mathrm{Age}_{E}(\delta_{i}),
    $$
    where $d_{\widetilde{\gamma}} = -K_{Y} \cdot \widetilde{\gamma}$ is the virtual dimension of the absolute invariant with the curve class $\beta$ (without marked points). Since $Y$ is Fano, we have $K_{Y} \cdot \widetilde{\gamma}_{2} < 0$ if $\widetilde{\gamma}_{2} > 0$. If $\delta_{i}$ is untwisted, then the intersection multiplicity $\mu_{i}$ is a positive integer. Hence the contribution at $\delta_{i}$ is $1 - \mu_{i} \leq 0$. If $\delta_{i}$ is twisted, then, since the singularity is terminal, the contribution at $\delta_{i}$ is 
    $$
    1 - \lfloor \mu_{i} \rfloor - (\{\mu_{i}\} + \mathrm{Age}_{E}(\delta_{i})) = 1 - \lfloor \mu_{i} \rfloor - \mathrm{Age}_{Y}(\delta_{i}) < 0.
    $$
    Hence we must have $\widetilde{\gamma}_{2}=0$, and all the intersections are untwisted and have multiplicities 1. The degeneration formula simplifies to
    $$
    \langle \tau_{-K_{Y}\cdot \widetilde{\gamma} -2}\mathbf{1} \rangle_{\widetilde{\gamma}}^{Y} = \sum\limits_{(\delta_{1},\cdots,\delta_{\rho})} \frac{\prod\limits_{j=1}^{\rho}\mu_{j}}{\rho!} \langle \tau_{-K_{Y}\cdot \widetilde{\gamma} -2} \mathbf{1} \mid \prod\limits_{j=1}^{\rho}\delta_{j} \rangle_{0,1+\rho,\widetilde{\gamma}}^{(Y,E)} \neq 0.
    $$
    Hence there exists a contact data $(\delta_{1},\cdots,\delta_{\rho})$ such that $\langle \tau_{-K_{Y}\cdot \widetilde{\gamma} -2} \mathbf{1} \mid \prod\limits_{j=1}^{\rho}\delta_{j} \rangle_{0,1+\rho,\widetilde{\gamma}}^{(Y,E)} \neq 0$. In particular, we have $E \cdot \widetilde{\gamma} \geq 0$.
\end{proof}

    The following lemma about the homology groups $H_{2}$ for terminal Fano threefolds is useful.

\begin{lemma}\label{lemma: H2 generated by curve classes}
    Let $X$ be a terminal Fano threefold. Then the homology group $H_{2}(X,\mathbb{Z})$ is generated by the class of algebraic curves.
\end{lemma}

\begin{proof}
    We can take a log resolution $\pi: \widetilde{X} \rightarrow X$ which is an isomorphism on the smooth locus of $X$. The variety $\widetilde{X}$ is a smooth uniruled threefold. Hence the integral Hodge conjecture holds on $\widetilde{X}$ (cf. \cite[Theorem 2]{Voisin2006IntegralHodgeClasses}). That is, every cohomology class in $H^{4}(\widetilde{X},\mathbb{Z}) \cap H^{2,2}(\widetilde{X})$ is generated by the class of algebraic curves. Since $X$ is Fano with rational singularities, by the Kawamata–Viehweg vanishing theorem, we have $H^{1}(\widetilde{X},\mathcal{O}_{\widetilde{X}}) = H^{1}(X,\mathcal{O}_{X})=0$ and $H^{2}(\widetilde{X},\mathcal{O}_{\widetilde{X}}) = H^{2}(X,\mathcal{O}_{X})=0$. Hence, by Serre duality, we have $H^{3,1}(\widetilde{X}) = H^{1,3}(\widetilde{X}) = 0$. In particular, we have $H^{4}(\widetilde{X},\mathbb{C}) = H^{2,2}(\widetilde{X})$. Hence the cohomology group $H^{4}(\widetilde{X},\mathbb{Z})$ is generated by the class of algebraic curves. By Poincaré duality, $H_{2}(\widetilde{X},\mathbb{Z})$ is also generated by the class of algebraic curves. Let $E$ be the exceptional divisor of $\pi$. Since threefold terminal singularities are isolated, we have $H_{1}(E,\mathbb{Z})=0$ by \cite[Theorem 1.1]{Tak03} and \cite[p. 930]{KK14}. Hence, by \cite[Lemma 4.2]{CasciniChen2024ChernNumbers}, the morphism $H_{2}(\widetilde{X},\mathbb{Z}) \rightarrow H_{2}(X,\mathbb{Z})$ is surjective. This shows that $H_{2}(X,\mathbb{Z})$ is also generated by the class of algebraic curves.
\end{proof}

\begin{theorem}\label{main theorem 1}
    Let $g: Y \rightarrow X$ be a divisorial contraction such that:
    \begin{enumerate}
        \item $X$ and $Y$ are Fano threefolds with $\mathbb{Q}$-factorial ordinary terminal singularities;
        \item $g$ contracts the exceptional divisor $E$ to a point $P \in X$;
        \item $P \in X$ is a smooth point, or a terminal quotient singularity, or a terminal singularity of type cA/n.
    \end{enumerate}
    Then we have
    $$
    \lim\limits_{r \rightarrow +\infty} \hat{G}_{Y,rE}(t) = \hat{G}_{X}(t)
    $$
    where $\hat{G}_{Y,rE}(t)$ and $\hat{G}_{X}(t)$ are the quantum periods of $(Y,rE)$ and $X$ respectively. In particular, if $Y$ has a parametrized toric Landau-Ginzburg model $\widetilde{f}_{Y}$ and $\lim\limits_{r \rightarrow +\infty} \widetilde{f}_{Y,rE}$ exists, then $X$ has a toric Landau-Ginzburg model $f_{X} = \lim\limits_{r \rightarrow +\infty} \widetilde{f}_{Y,rE}$.
\end{theorem}
    
\begin{proof}
    By assumption $g$ contracts the exceptional divisor $E$ to a point $P \in X$ such that $P \in X$ is a smooth point or a quotient point or a terminal singularity of type cA/n. Hence by \cref{thm: degeneration to weighted normal cone}, we can construct a family $\mathcal{W} \rightarrow C \ni 0$ such that:
    \begin{enumerate}
        \item The general fibre $\mathcal{W}_{t}$ is a $\mathbb{Q}$-smoothing of $X$;
        \item The special fibre $\mathcal{W}_{0} = Y' \cup_{E'} Y''$ is a transversal intersection between smooth Deligne-Mumford stacks.
        \item $(Y',E')$ is a $\mathbb{Q}$-smoothing of $(Y,E)$.
        \item The total space $\mathcal{W}$ has quotient singularities.
    \end{enumerate} 
    We can replace $(Y,E)$ by $(Y',E')$ and $X$ by the image of the divisorial contraction of $E'$, and assume that $(Y,E)$ is already a $\mathbb{Q}$-smoothing of itself.
    
    By \cref{lemma: limit of quantum period} and the definition of the quantum period, we have
    $$
    \lim\limits_{r \rightarrow +\infty} \hat{G}_{Y,rE}(t) = 1 + \sum\limits_{\substack{\beta \in H_{2}(Y,\mathbb{Z}) \\ E \cdot \beta = 0}}(-K_{Y} \cdot \beta)!\langle \tau_{-K_{Y}\cdot\beta -2}\mathbf{1} \rangle_{\beta}^{Y} \cdot t^{-K_{Y} \cdot \beta}.
    $$
    On the other side, the quantum period of $X$ is given by
    $$
    \hat{G}_{X}(t) = 1 + \sum\limits_{\beta \in H_{2}(X,\mathbb{Z})}(-K_{X} \cdot \beta)!\langle \tau_{-K_{X}\cdot\beta -2}\mathbf{1} \rangle_{\beta}^{X} \cdot t^{-K_{X} \cdot \beta}.
    $$
    Since $X$ and $Y$ are terminal Fano threefolds, the singular homology groups $H_{2}(X,\mathbb{Z})$ and $H_{2}(Y,\mathbb{Z})$ are generated by algebraic 1-cycles by \cref{lemma: H2 generated by curve classes}. Since $g$ contracts a divisor to a point, the natural pushforward $H_{2}(Y,\mathbb{Z}) \rightarrow H_{2}(X,\mathbb{Z})$ is surjective. Now we show that $\lim\limits_{r \rightarrow +\infty} \hat{G}_{Y,rE}(t) = \hat{G}_{X}(t)$.

    Let $\beta \in H_{2}(X,\mathbb{Z})$ be a 1-cycle on $X$ such that $\langle \tau_{-K_{X}\cdot \beta -2}\mathbf{1} \rangle_{\beta} \neq 0$. The point cycle $\mathbf{1}$ on $\mathcal{W}_{t}$ lifts to the cycle $(\mathbf{1},0)$ on $\mathcal{W}_{0}$. We apply \cref{theorem: orbifold degeneration formula} to the family $\mathcal{W} \rightarrow C \ni 0$. The topological data $(g,n,\beta)$ lifts to two admissible triples $\Gamma_{1}$ on $(Y,E)$ and $\Gamma_{2}$ on $(Y'',E)$ such that $\Gamma_{1}$ has curve class $\widetilde{\gamma}$, contact order $\mu = (E \cdot \widetilde{\gamma})$ and number of contact points $\rho$. Write $\pi^{*}K_{X} = K_{Y}-aE$ for some positive real number $a$. We have $K_{Y} \cdot \widetilde{\gamma} = K_{X} \cdot \beta + a \mu$. Since the point class $\mathbf{1}$ lies in the untwisted sector, it suffices to consider the virtual class $[\mathcal{K}_{0,1}(X,\beta)^{\mathrm{untw}}]^{vir}$ of the untwisted component. Hence we can assume that the age at the marked point is zero. Let $(\delta_{1},\cdots,\delta_{\rho})$ be the contact data and $\mu_{i} = (E \cdot \widetilde{\gamma})_{\delta_{i}}$ be the associated intersection multiplicity. Then the relative virtual dimension is given by
    $$
    d_{\Gamma_{1}} = - K_{Y} \cdot \widetilde{\gamma} + \rho - \mu - \sum\limits_{i=1}^{\rho} \mathrm{Age}_{E}(\delta_{i}) = d_{\beta} + \rho - (1 + a)\mu - \sum\limits_{i=1}^{\rho} \mathrm{Age}_{E}(\delta_{i})
    $$
    where $d_{\beta} = -K_{X} \cdot \beta$ is the virtual dimension of the absolute invariant with curve class $\beta$ (without marked points). If $\delta_{i}$ is untwisted, then the intersection multiplicity $\mu_{i}$ is a positive integer. Hence the contribution at $\delta_{i}$ is $1 - (1+a)\mu_{i} < 0$. If $\delta_{i}$ is twisted, then, since the singularity is terminal, the contribution at $\delta_{i}$ is $$
    1 - (1+a)\lfloor \mu_{i} \rfloor - a\{\mu_{i}\} - (\{\mu_{i}\} + \mathrm{Age}_{E}(\delta_{i})) = 1 - (1+a)\lfloor \mu_{i} \rfloor - a\{\mu_{i}\} - \mathrm{Age}_{Y}(\delta_{i}) < 0.
    $$
    Hence if $\rho > 0$ then the contribution from the relative part to the virtual dimension is always negative. Hence to have the correct virtual dimension, we must have $\rho = 0$. Hence we have
    $$
    \langle \tau_{-K_{X}\cdot \beta -2}\mathbf{1} \rangle_{\beta}^{X} = \sum\limits_{\widetilde{\gamma}} \langle \tau_{-K_{Y}\cdot\widetilde{\gamma} -2}\mathbf{1} \mid \emptyset \rangle_{\widetilde{\gamma}}^{(Y,E)}
    $$
    where the sum is taken over classes $\widetilde{\gamma}$ on 
    $Y \cup_{E} Y''$ such that
    $$
    g_{*}\widetilde{\gamma} = \beta, \qquad E \cdot \widetilde{\gamma} = 0, \qquad \widetilde{\gamma}_{Y''} = 0.
    $$
    In particular, there exists a 1-cycle class $\widetilde{\gamma}$ such that $g_{*}\widetilde{\gamma} = \beta$. By the negativity lemma, the 1-cycle class $\widetilde{\gamma}$ satisfying the above conditions is unique. Hence we have $\lim\limits_{r \rightarrow +\infty} \hat{G}_{Y,rE}(t) = \hat{G}_{X}(t)$.
    
    Now assume that $Y$ has a parametrized toric Landau-Ginzburg model $\widetilde{f}_{Y}$ and $\lim\limits_{r \rightarrow +\infty} \widetilde{f}_{Y,rE}$ exists. Write $f' = \lim\limits_{r \rightarrow +\infty} \widetilde{f}_{Y,rE}$. We claim that $f'$ is a toric LG model of $X$. Indeed, the period condition follows immediately from the equation $\lim\limits_{r \rightarrow +\infty} \hat{G}_{Y,rE}(t) = \hat{G}_{X}(t)$. It follows that the corresponding degeneration of the Calabi-Yau compactification of $f$ is irreducible. Hence the degeneration is a Calabi-Yau compactification of $f'$. Finally, the toric condition is trivial.
\end{proof}

\begin{theorem}\label{main theorem 2}
    Let $g: Y \rightarrow X$ be a divisorial contraction such that:
    \begin{enumerate}
        \item $X$ and $Y$ are Fano threefolds with $\mathbb{Q}$-factorial ordinary terminal singularities;
        \item $g$ contracts the exceptional divisor $E$ to a smooth curve $C$ on $X$;
        \item $C$ only contains terminal singularities of type cA or cA/n.
    \end{enumerate}
    Then we have
    $$
    \lim\limits_{r \rightarrow +\infty} \hat{G}_{Y,rE}(t) = \hat{G}_{X}(t)
    $$
    where $\hat{G}_{Y,rE}(t)$ and $\hat{G}_{X}(t)$ are the quantum periods of $(Y,rE)$ and $X$ respectively. In particular, if $Y$ has a parametrized toric Landau-Ginzburg model $\widetilde{f}_{Y}$ and $\lim\limits_{r \rightarrow +\infty} \widetilde{f}_{Y,rE}$ exists, then $X$ has a toric Landau-Ginzburg model $f_{X} = \lim\limits_{r \rightarrow +\infty} \widetilde{f}_{Y,rE}$.
\end{theorem}

\begin{proof}
    The proof is similar to the proof of \cref{main theorem 1}. Indeed, by assumption $g$ contracts the exceptional divisor $E$ to a smooth curve $C \subset X$ such that $C$ only contains terminal singularities of type cA or type cA/n. Hence by \cref{thm: degeneration to weighted normal cone 2}, we can construct a family $\mathcal{W} \rightarrow B \ni 0$ such that:
    \begin{enumerate}
        \item The general fibre $\mathcal{W}_{t}$ is a $\mathbb{Q}$-smoothing of $X$;
        \item The special fibre $\mathcal{W}_{0} = Y' \cup_{E'} Y''$ is a transversal intersection between (disjoint unions of) smooth Deligne-Mumford stacks.
        \item $(Y',E')$ contains a $\mathbb{Q}$-smoothing of $(Y,E)$.
        \item The total space $\mathcal{W}$ has quotient singularities.
    \end{enumerate} 
    We can replace $(Y,E)$ by its $\mathbb{Q}$-smoothing and $X$ by the image of the divisorial contraction of the $\mathbb{Q}$-smoothing of $(Y,E)$, and assume that $(Y,E)$ is already a $\mathbb{Q}$-smoothing of itself.

    By \cref{lemma: limit of quantum period} and the definition of the quantum period, we have
    $$
    \lim\limits_{r \rightarrow +\infty} \hat{G}_{Y,rE}(t) = 1 + \sum\limits_{\substack{\beta \in H_{2}(Y,\mathbb{Z}) \\ E \cdot \beta = 0}}(-K_{Y} \cdot \beta)!\langle \tau_{-K_{Y}\cdot\beta -2}\mathbf{1} \rangle_{\beta}^{Y} \cdot t^{-K_{Y} \cdot \beta}.
    $$
    On the other side, the quantum period of $X$ is given by
    $$
    \hat{G}_{X}(t) = 1 + \sum\limits_{\beta \in H_{2}(X,\mathbb{Z})}(-K_{X} \cdot \beta)!\langle \tau_{-K_{X}\cdot\beta -2}\mathbf{1} \rangle_{\beta}^{X} \cdot t^{-K_{X} \cdot \beta}.
    $$
    Since $X$ and $Y$ are terminal Fano threefolds, the singular homology groups $H_{2}(X,\mathbb{Z})$ and $H_{2}(Y,\mathbb{Z})$ are generated by algebraic 1-cycles by \cref{lemma: H2 generated by curve classes}. Hence the natural morphism $H_{2}(Y,\mathbb{Z}) \rightarrow H_{2}(X,\mathbb{Z})$ is surjective by Tsen's theorem. Now we show that $\lim\limits_{r \rightarrow +\infty} \hat{G}_{Y,rE}(t) = \hat{G}_{X}(t)$.
    
    Let $\beta \in H_{2}(X,\mathbb{Z})$ be a 1-cycle on $X$ such that $\langle \tau_{-K_{X}\cdot \beta -2}\mathbf{1} \rangle_{\beta} \neq 0$. The point cycle $\mathbf{1}$ on $\mathcal{W}_{t}$ lifts to the cycle $(\mathbf{1},0)$ on $\mathcal{W}_{0}$. We apply \cref{theorem: orbifold degeneration formula} to the family $\mathcal{W} \rightarrow B \ni 0$. The topological data $(g,n,\beta)$ lifts to two admissible triples $\Gamma_{1}$ on $(Y,E)$ and $\Gamma_{2}$ on $(Y'',E)$ such that $\Gamma_{1}$ has curve class $\widetilde{\gamma}$, contact order $\mu = (E \cdot \widetilde{\gamma})$ and number of contact points $\rho$. Write $\pi^{*}K_{X} = K_{Y}-aE$ for some positive real number $a$. We have $K_{Y} \cdot \widetilde{\gamma} = K_{X} \cdot \beta + a \mu$. Since the point class $\mathbf{1}$ lies in the untwisted sector, it suffices to consider the virtual class $[\mathcal{K}_{0,1}(X,\beta)^{\mathrm{untw}}]^{vir}$ of the untwisted component. Hence we can assume that the age at the marked point is zero. Let $(\delta_{1},\cdots,\delta_{\rho})$ be the contact data and $\mu_{i} = (E \cdot \widetilde{\gamma})_{\delta_{i}}$ be the associated intersection multiplicity. Then the relative virtual dimension is given by
    $$
    d_{\Gamma_{1}} = - K_{Y} \cdot \widetilde{\gamma} + \rho - \mu - \sum\limits_{i=1}^{\rho} \mathrm{Age}_{E}(\delta_{i}) = d_{\beta} + \rho - (1 + a)\mu - \sum\limits_{i=1}^{\rho} \mathrm{Age}_{E}(\delta_{i})
    $$
    where $d_{\beta} = -K_{X} \cdot \beta$ is the virtual dimension of the absolute invariant with curve class $\beta$ (without marked points). If $\delta_{i}$ is untwisted, then the intersection multiplicity $\mu_{i}$ is a positive integer. Hence the contribution at $\delta_{i}$ is $1 - (1+a)\mu_{i} < 0$. If $\delta_{i}$ is twisted, then, since the singularity is terminal, the contribution at $\delta_{i}$ is $$
    1 - (1+a)\lfloor \mu_{i} \rfloor - a\{\mu_{i}\} - (\{\mu_{i}\} + \mathrm{Age}_{E}(\delta_{i})) = 1 - (1+a)\lfloor \mu_{i} \rfloor - a\{\mu_{i}\} - \mathrm{Age}_{Y}(\delta_{i}) < 0.
    $$
    Hence if $\rho > 0$ then the contribution from the relative part to the virtual dimension is always negative. Hence to have the correct virtual dimension, we must have $\rho = 0$. Hence we have
    $$
    \langle \tau_{-K_{X}\cdot \beta -2}\mathbf{1} \rangle_{\beta}^{X} = \sum\limits_{\widetilde{\gamma}} \langle \tau_{-K_{Y}\cdot\widetilde{\gamma} -2}\mathbf{1} \mid \emptyset \rangle_{\widetilde{\gamma}}^{(Y,E)}
    $$
    where the sum is taken over classes $\widetilde{\gamma}$ on 
    $Y \cup_{E} Y''$ such that
    $$
    g_{*}\widetilde{\gamma} = \beta, \qquad E \cdot \widetilde{\gamma} = 0, \qquad \widetilde{\gamma}_{Y''} = 0.
    $$
    In particular, there exists a 1-cycle class $\widetilde{\gamma}$ such that $g_{*}\widetilde{\gamma} = \beta$. By the negativity lemma, the 1-cycle class $\widetilde{\gamma}$ satisfying the above conditions is unique. Hence we have $\lim\limits_{r \rightarrow +\infty} \hat{G}_{Y,rE}(t) = \hat{G}_{X}(t)$.
    
    Now assume that $Y$ has a parametrized toric Landau-Ginzburg model $\widetilde{f}_{Y}$ and $\lim\limits_{r \rightarrow +\infty} \widetilde{f}_{Y,rE}$ exists. Write $f' = \lim\limits_{r \rightarrow +\infty} \widetilde{f}_{Y,rE}$. We claim that $f'$ is a toric LG model of $X$. Indeed, the period condition follows immediately from the equation $\lim\limits_{r \rightarrow +\infty} \hat{G}_{Y,rE}(t) = \hat{G}_{X}(t)$. It follows that the corresponding degeneration of the Calabi-Yau compactification of $f$ is irreducible. Hence the degeneration is a Calabi-Yau compactification of $f'$. Finally, the toric condition is trivial.
\end{proof}

\section{Applications}

    In this section, we study an application of our result. The following example shows that how to use a Sarkisov link to compute toric LG models.

\begin{example}[A toric LG model of the terminal Fano threefold GRDB\#40836]\label{example: toric LG model of GRDB 40836}
    Let $X$ be the $\mathbb{Q}$-Fano threefold given by GRDB\#40836. By \cite[Example 6.3.3]{ProkhorovReid2016QFanoIndex2} we have the Sarkisov link
    $$
    \begin{tikzcd}
        & Y \arrow[ld,swap,"\pi_{1}"] \arrow[rd,"\pi_{2}"] & \\
        \mathbb{P}^{3} & & X 
    \end{tikzcd}
    $$
    where
    \begin{enumerate}
        \item $Y$ has a unique singular point, which is a terminal quotient singularity of type $\frac{1}{2}(1,1,1)$, and $X$ has a unique singular point, which is a terminal quotient singularity of type $\frac{1}{3}(1,1,2)$.
        \item $\pi_{1}$ is the symbolic blow-up of a curve $\Gamma$ of degree 7 in some quadric cone $Q$, such that $\Gamma$ has a singularity of order 3 at the vertex of $Q$. Let $E$ be the exceptional divisor of $\pi_{1}$.
        \item $\pi_{2}$ is the Kawamata blow-up of the singularity $\frac{1}{3}(1,1,2)$ on $X$. The exceptional divisor of $\pi_{2}$ is the strict transformation $\widetilde{Q}$ of $Q$ on $Y$.
    \end{enumerate}
    In particular, \cref{main theorem 1} can be applied to the divisorial contraction $\pi_{2}$. First, we construct a parametrized toric LG model of $Y$ by degenerating the birational morphism $\pi_{1}$. Let $[x,y,z,w]$ be the homogeneous coordinates of $\mathbb{P}^3$. Take a degeneration of the quadric cone $Q$ to a union of two hyperplanes $H_{1}=(y=0)$ and $H_{2}=(z=0)$, such that the vertex degenerates to $O = (x=y=z=0)$. The curve $\Gamma$ degenerates to a union of 7 lines given by 
    \begin{align*}
    L_{1} = (x=y=0),\ L_{2} = (x=z=0),\ L_{3}=(y=z=0), \\ 
    L_{4}=(y=w=0),\ L_{5}=(z=w=0),
    \end{align*}
    and two lines $L_{6} = (y = w + \epsilon x + \epsilon z = 0)$ and $L_{7} = (z = w + \epsilon x + \epsilon y = 0)$ for some positive real number $\epsilon$. The union $\bigcup\limits_{i=1}^{7}L_{i}$ has three ordinary triple points $L_{1} \cap L_{2} \cap L_{3}$, $L_{3}\cap L_{4} \cap L_{5}$, and $L_{3} \cap L_{6} \cap L_{7}$, and several ordinary double points. We are now ready to construct a toric degeneration of $Y$. Take the standard fan of $\mathbb{P}^3$ generated by vectors $(1,0,0), (0,1,0), (0,0,1), (-1,-1,-1)$, which correspond to the coordinates $x,y,z,w$ respectively. Take the degeneration of $Q$ and $\Gamma$ as described above. Then:
    \begin{enumerate}
        \item At the toric invariant point $(x=y=z=0)$, the symbolic blow-up of $\Gamma$ degenerated to the symbolic blow-up of $L_{1} \cup L_{2} \cup L_{3}$, which is given by the subdivision of the cone $\langle(1,0,0), (0,1,0), (0,0,1)\rangle$ into four cones 
        \begin{gather*}
        \langle(1,0,0), (1,1,0), (1,0,1)\rangle, \\
        \langle(0,1,0), (1,1,0), (0,1,1)\rangle, \\
        \langle(0,0,1), (1,0,1), (0,1,1)\rangle, \\
        \langle(1,1,0), (1,0,1), (0,1,1)\rangle.
        \end{gather*}
        \item At the toric invariant point $(y=z=w=0)$, the symbolic blow-up of $\Gamma$ degenerated to the \emph{ordinary} blow-up of $L_{3} \cup L_{4} \cup L_{5}$, which is given by the subdivision of the cone $\langle (0,1,0), (0,0,1), (-1,-1,-1)\rangle$ into three cones 
        \begin{gather*}
        \langle(0,1,0), (0,1,1), (-1,0,0),(-1,0,-1)\rangle, \\
        \langle(0,0,1), (0,1,1), (-1,0,0),(-1,-1,0)\rangle, \\
        \langle(-1,-1,-1), (-1,-1,0), (-1,0,0),(-1,0,-1)\rangle.
        \end{gather*}
        \item The two lines $L_{6}$ and $L_{7}$ are not toric-invariant. We further degenerate $L_{6} \cup L_{7}$ by letting $\epsilon \rightarrow 0$, into a union of three infinitely near irreducible toric-invariant curves. The blow-up is given by the subdivision of the cone $\langle(0,1,0), (0,1,1), (-1,0,0),(-1,0,-1)\rangle$ into two cones 
        \begin{gather*}
        \langle(0,1,0), (-1,1,-1), (-1,1,1),(0,1,1)\rangle, \\
        \langle(-1,1,-1), (-1,0,-1), (-1,0,0),(-1,1,1)\rangle,
        \end{gather*}
        and the subdivision of the cone $\langle(0,0,1), (0,1,1), (-1,0,0),(-1,-1,0)\rangle$ into two cones 
        \begin{gather*}
        \langle(0,0,1), (-1,-1,1), (-1,1,1),(0,1,1)\rangle, \\
        \langle(-1,-1,1), (-1,-1,0), (-1,0,0),(-1,1,1)\rangle.
        \end{gather*}
    \end{enumerate}
    
    Hence $Y$ degenerates to the toric canonical Fano threefold GRDB(toric canonical threefolds)\#372945. Let $\pi': Y' \rightarrow Y$ be the Kawamata blow-up of the singularity $\frac{1}{2}(1,1,1)$. On the degenerated toric canonical Fano threefold GRDB(toric canonical threefolds)\#372945, the morphism $\pi'$ is given by the subdivision of the cone $\langle(1,1,0), (1,0,1), (0,1,1)\rangle$ into three cones 
    \begin{gather*}
        \langle(1,1,0), (1,0,1), (1,1,1)\rangle, \\
        \langle(1,1,0), (0,1,1), (1,1,1)\rangle, \\
        \langle(1,0,1), (0,1,1), (1,1,1)\rangle.
    \end{gather*}
    Hence $Y'$ is a smooth weak Fano threefold of Picard number 3. Now we embed $Y'$ as a codimension 3 toric complete intersection. Indeed, set
    \begin{gather*}
    \rho_{1} = (1,0,0,0,0,0), \\
    \rho_{2} = (0,1,0,0,0,0), \\
    \rho_{3} = (0,0,1,0,0,0), \\
    \rho_{4} = (0,0,0,1,0,0), \\
    \rho_{5} = (0,0,0,0,1,0), \\
    \rho_{6} = (0,0,0,0,0,1), \\
    \rho_{7} = (0,-1,0,-1,0,-1), \\
    \rho_{8} = (0,0,0,-1,0,-1), \\
    \rho_{9} = (-1,1,-1,-1,-1,-1).
    \end{gather*}
    Let $D_{i}$ be the corresponding toric invariant divisor and set the nef partition
    $$
    L_{1} = D_{1}+D_{2}, \quad L_{2} = D_{3}+D_{4}, \quad L_{3}=D_{5}+D_{6}.
    $$
    The effective cone $\mathrm{Eff}_{\mathbb{R}}(Y')$ is generated by 3 divisors:
    \begin{enumerate}
        \item The exceptional divisor $E'$ of $\pi'$;
        \item The strict birational transformation $\widetilde{E}$ of the exceptional divisor $E$ of $\pi_{1}$;
        \item The strict birational transformation of the exceptional divisor $\pi_{1}^{*}Q - E$ of $\pi_{2}$.
    \end{enumerate}
    Using Givental's formula, the constant term of the regularized $I$-series of $Y'$ is computed by the period of the Laurent polynomial
    $$
    \widetilde{f}_{Y'} = x + a_{1}(y+z) + a_{2}(xy+xz+a_{1}yz) + \frac{(1+a_{2}y)^{2}(1+a_{2}z)^{2}}{xyz} + a_{3}xyz.
    $$
    The morphism $\pi'$ contracts a smooth divisor on a smooth weak Fano threefold to a quotient singularity of type $\frac{1}{2}(1,1,1)$. Notice that $Y'$ is weak Fano, so the $I$-series and the $J$-series are different, and it may not have a well-defined quantum period since there is a $K_{Y'}$-trivial curve class $\widetilde{\gamma}_{0}$. Nonetheless, one can show, as in \cite{Self_LG_Model}, that to compute the quantum period of $Y$, it suffices to consider the curve classes $\widetilde{\gamma}$ such that $E' \cdot \widetilde{\gamma} = 0$, so the change of variables has no contribution to the quantum period of $Y$. In particular, we can compute the quantum period of $Y$ by the $I$-series of $Y'$. Setting $a_{3} \to 0$, we have
    $$
    \widetilde{f}_{Y} = x + a_{1}(y+z) + a_{2}(xy+xz+a_{1}yz) + \frac{(1+a_{2}y)^{2}(1+a_{2}z)^{2}}{xyz}.
    $$
    Hence, by \cref{main theorem 1}, setting $a_{1} \to 0$ and $a_{2} = 1$, we obtain a toric LG model of $X$ given by
    $$
    f_{X} = x  + xy+xz + \frac{(1+y)^{2}(1+z)^{2}}{xyz}.
    $$
    The corresponding toric canonical Fano threefold is GRDB(toric canonical threefolds)\#491680.
\end{example}

\printbibliography
\end{document}